\documentclass[a4paper, reqno]{article}
\usepackage[english]{babel}

\usepackage[usenames,dvipsnames]{xcolor}

\usepackage{enumerate}
\usepackage[utf8]{inputenc}
\usepackage[left=3cm,right=2cm,top=3cm,bottom=2cm]{geometry}
\usepackage{amsthm, amssymb, mathtools, hyperref, cases, multicol, enumitem}
\usepackage{graphicx}
\usepackage{indentfirst}

\numberwithin{equation}{section}

\newtheorem{thm}{Theorem}[section]

\newtheorem{prop}[thm]{Proposition}
\newtheorem{lem}[thm]{Lemma}
\theoremstyle{remark}

\newtheorem{rmk}[thm]{Remark}

\theoremstyle{definition}

\newtheorem{cond}[thm]{Condition}

\newcommand{\T}{\mathcal{T}}
\newcommand{\eps}{\varepsilon}

\DeclareMathOperator{\interior}{int}

\newcommand{\real}{\mathbb{R}}
\newcommand{\nat}{\mathbb{N}}
\newcommand{\A}{\mathcal{A}}
\newcommand{\N}{\mathcal{N}}

\newcommand{\dif}{\mathrm{d}}

\DeclarePairedDelimiter{\abs}{\lvert}{\rvert}
\DeclarePairedDelimiter{\norm}{\lVert}{\rVert}
\DeclarePairedDelimiter{\parens}{(}{)}
\DeclarePairedDelimiter{\set}{\{}{\}}

\DeclarePairedDelimiter{\coi}{\lbrack}{\lbrack}
\DeclarePairedDelimiter{\oci}{\rbrack}{\rbrack}
\DeclarePairedDelimiter{\ooi}{\rbrack}{\lbrack}

\title{Symmetry and classification of positive standing waves of nonlinear Hartree type equations}
\author{
Eduardo de Souza Böer
\and
Ederson Moreira dos Santos
\and
Gustavo de Paula Ramos
}

\begin{document}

\maketitle

\begin{abstract}
This paper presents some qualitative properties of positive solutions to the strongly coupled system
\[
\begin{cases}
\displaystyle
- \Delta u + \tau u
=
\frac{2 p}{p + q}
\left( I_\alpha \ast |v|^q \right)
|u|^{p - 2} u
&\text{in} ~ \mathbb{R}^N,
\\
\\
\displaystyle
- \Delta v + \eta v
=
\frac{2 q}{p + q}
\left( I_\alpha \ast |u|^p \right)
|v|^{q - 2} v
&\text{in} ~ \mathbb{R}^N,
\end{cases}
\]
with $\tau, \eta > 0$, $N \in \mathbb{N}$, $0 < \alpha < N$,
\[
\max \left\{1, \frac{2 \alpha}{N}\right\} < p, q < 2^*
\quad \text{and} \quad
\frac{2 (N + \alpha)}{N} < p + q < 2_\alpha^*,
\]
where $I_\alpha$ denotes the Riesz potential,
\[
2^*
:=
\begin{cases}
\infty,
&\text{if} ~ N \in \{1, 2\},
\\
\frac{2 N}{N - 2},
&\text{if} ~ N \geq 3,
\end{cases}
\quad \text{and} \quad
2_\alpha^*
:=
\begin{cases}
\infty,
&\text{if} ~ N \in \{1, 2\},
\\
\frac{2 (N + \alpha)}{N - 2},
&\text{if} ~ N \geq 3.
\end{cases}
\]
More precisely, by means of the moving planes method, we prove that positive solutions to this system are radially symmetric and strictly radially decreasing when $p, q \geq 2$, and we obtain a classification result  for positive ground states in the case $p = q$ and $\tau = \eta$.
\end{abstract}


\section{Introduction}
This work focuses on qualitative properties of positive standing waves of the Hartree type equations
\begin{equation}
\label{eqn:Hartree-system}
\begin{cases}
\displaystyle
- i \frac{\partial}{\partial t} \varphi
=
\Delta \varphi
+
\frac{2 p}{p + q}
\parens*{I_\alpha \ast |\psi|^q}
|\varphi|^{p - 2} \varphi
&\text{in} ~ \real \times \real^N,
\\
\\
\displaystyle
- i \frac{\partial}{\partial t} \psi
=
\Delta \psi
+
\frac{2 q}{p + q}
\parens*{I_\alpha \ast |\varphi|^p}
|\psi|^{q - 2} \psi
&\text{in} ~ \real \times \real^N,
\end{cases}
\end{equation}
where $N \in \nat$, $0 < \alpha < N$,
$I_\alpha \colon \real^N \setminus \set{0} \to \real$
denotes the \emph{Riesz potential} defined as
\[
I_\alpha \parens{x}
=
\frac
	{\Gamma \parens*{\frac{N - \alpha}{2}}}
	{
		\Gamma \parens*{\frac{\alpha}{2}}
		\pi^{\frac{N}{2}}
		2^\alpha
	}
\frac{1}{|x|^{N - \alpha}}\, ,
\]
$p, q > 1$
and the unknowns are
$\varphi, \psi \colon \real \times \real^N \to \mathbb{C}$.

A routine verification shows that standing waves of the form
$\varphi \parens{t, x} = u \parens{x} e^{\mathrm{i} \tau t}$,
$\psi \parens{t, x} = v \parens{x} e^{\mathrm{i} \eta t}$
solve \eqref{eqn:Hartree-system} precisely when
\begin{equation}
\label{eqn:Hartree-system:2}
\begin{cases}
\displaystyle
- \Delta u + \tau u
=
\frac{2 p}{p + q}
\parens*{I_\alpha \ast |v|^q}
|u|^{p - 2} u
&\text{in} ~ \real^N,
\\
\\
\displaystyle
- \Delta v + \eta v
=
\frac{2 q}{p + q}
\parens*{I_\alpha \ast |u|^p}
|v|^{q - 2} v
&\text{in} ~ \real^N.
\end{cases}
\end{equation}
In the case $\tau = \eta$, $p = q$ and $u = v$, \eqref{eqn:Hartree-system:2} becomes the Choquard equation
\begin{equation}
\label{eqn:Choquard-type}
- \Delta w + \tau w
=
\parens*{I_\alpha \ast \abs{w}^p}
\abs{w}^{p - 2} w
\quad \text{in} \quad
\real^N.
\end{equation}

This equation appears in many different contexts as a model for nonlocal self-attractive effects, especially in the case $N = 3$ and $p = \alpha = 2$; see \cite{doi:10.1098/rspa.1937.0106, frohlichElectronsLatticeFields1954, pekarUntersuchungenUberElektronentheorie1954, penroseGravitysRoleQuantum1996}. From a mathematical point of view, it is already relatively well-understood and we refer the reader to \cite{morozGroundstatesNonlinearChoquard2013, morozGuideChoquardEquation2017} for an overview about it.

There are still few studies about \eqref{eqn:Hartree-system:2} and related systems. For instance, \cite{bhattaraiExistenceStabilityStanding2019} established the existence and orbital stability of ground states of the associated normalized system with $N \in \set{2, 3}$ and $p = q$. Similar systems were considered in \cite{wangStandingWavesCoupled2017}. The symmetry and classification of positive solutions to \eqref{eqn:Hartree-system:2} in the case
$\tau = \eta = 0$ was already established in \cite{wangClassificationQualitativeAnalysis2024}. A systematic study of \eqref{eqn:Hartree-system:2} in the case $\tau = \eta = 1$ was recently done in \cite{boerStandingWavesNonlinear2025}, including (i) conditions for the existence of positive ground states, (ii) symmetry of positive ground states, (iii) regularity/integrability of weak solutions, (iv) decay at infinity for positive ground states and (v) a nonexistence result.

Inspired by the setting considered in \cite{boerStandingWavesNonlinear2025}, we suppose throughout this paper that the following conditions on the parameters $p, q$ are satisfied:
\begin{equation}
\tag{H$_1$}
\label{H1}
\max \left\{1, \frac{2 \alpha}{N}\right\} < p, q < 2^*
\quad \text{and} \quad
\frac{2 \parens{N + \alpha}}{N} < p + q < 2_\alpha^*,
\end{equation}
where
\[
2^*
:=
\begin{cases}
\infty,
&\text{if} ~ N \in \{1, 2\},
\\
\frac{2 N}{N - 2},
&\text{if} ~ N \geq 3,
\end{cases}
\quad \text{and} \quad
2_\alpha^*
:=
\begin{cases}
\infty,
&\text{if} ~ N \in \{1, 2\},
\\
\frac{2 (N + \alpha)}{N - 2},
&\text{if} ~ N \geq 3.
\end{cases}
\]
In particular, see \cite[p. 8--9]{boerStandingWavesNonlinear2025}, these conditions ensure that there exist
$\theta_1, \theta_2 \in \ooi{1, \frac{N}{\alpha}}$
such that
\begin{equation}
\label{eqn:thetas}
\frac{1}{\theta_1} + \frac{1}{\theta_2} = \frac{N + \alpha}{N}
\quad \text{and} \quad
2 < \theta_1 p, \theta_2 q < 2^*.
\end{equation}

In this paper we consider weak solutions to \eqref{eqn:Hartree-system:2} in the Hilbert space
\[
E \parens{\real^N}
:=
H^1 \parens{\real^N} \times H^1 \parens{\real^N}.
\]
More precisely,
$\parens{u, v} \in E \parens{\real^N}$
is said to be a \emph{weak solution} to \eqref{eqn:Hartree-system:2} when
\[
\int_{\real^N}
	\nabla u \cdot \nabla \varphi
	+
	\tau u \varphi
\dif \mu
=
\frac{2 p}{p + q}
\int_{\real^N}
	\parens*{I_\alpha \ast |v|^q}
	|u|^{p - 2}
	\varphi
\dif \mu
\]
and
\[
\int_{\real^N}
	\nabla v \cdot \nabla \xi
	+
	\eta v \xi
\dif \mu
=
\frac{2 q}{p + q}
\int_{\real^N}
	\parens*{I_\alpha \ast |u|^p}
	|v|^{q - 2}
	\xi
\dif \mu
\]
for every
$\varphi, \xi \in C_c^\infty \parens{\real^N}$ (and hence for all $\varphi, \xi \in H^1\parens{\real^N}$). It follows from \cite[Theorem 1.4]{boerStandingWavesNonlinear2025} that weak solutions in
$E \parens{\real^N}$ to \eqref{eqn:Hartree-system:2} are actually classical solutions to \eqref{eqn:Hartree-system:2}. As such, we will henceforth refer to them simply as \emph{solutions in $E \parens{\real^N}$ to \eqref{eqn:Hartree-system:2}}. In this context, the aforementioned numbers $\theta_1, \theta_2$ can then be used with the Hardy--Littlewood--Sobolev inequality (see Proposition \ref{prop:HLS}) and the Sobolev embeddings to obtain a variational characterization for solutions in $E \parens{\real^N}$ to \eqref{eqn:Hartree-system:2}; see Section \ref{intro:classification} ahead.

Finally, the goal of this paper is to deepen the investigation started in \cite{boerStandingWavesNonlinear2025} by examining qualitative properties of positive solutions to \eqref{eqn:Hartree-system:2}. More precisely,  we establish a result of symmetry for positive solutions, and a classification result for positive ground states in the case
$p = q$ and $\tau = \eta$.

\subsection{Symmetry of positive solutions} Due to the nonlocal terms, the classical moving plane arguments \cite{GidasNiNirenberg} for partial differential equations do not apply to prove symmetry for positive solutions to \eqref{eqn:Hartree-system:2}.  We prove such property in Section \ref{symmetry} by means of the Chen--Li--Ou integral form of the moving planes method, introduced in \cite{chenClassificationSolutionsSystem2005, chenClassificationSolutionsIntegral2006}, which has been used in \cite[Section 3]{vairaExistenceBoundStates2013}, \cite[Section 7]{wangStandingWavesCoupled2017} and \cite[Section 4]{wangClassificationQualitativeAnalysis2024} for some related problems. The first main result in this paper reads as follows.

\begin{thm}
\label{thm:symmetry}
Suppose that $\tau, \eta > 0$, $N \in \nat$, $0 < \alpha < N$,
\begin{equation}
\label{H2}
\tag{H$_2$}
2 \leq p, q < 2^*,
\quad
p + q < 2_\alpha^*
\end{equation}
and
$\parens{u, v} \in E \parens{\real^N}$
is a positive solution to \eqref{eqn:Hartree-system:2}. Then there exists $x_0 \in  \real^N$ such that $u (\cdot - x_0)$ and
$v (\cdot - x_0)$ are radial and strictly radially decreasing.
\end{thm}

\begin{figure}[!h]
\begin{center}
 \includegraphics[width=7cm]{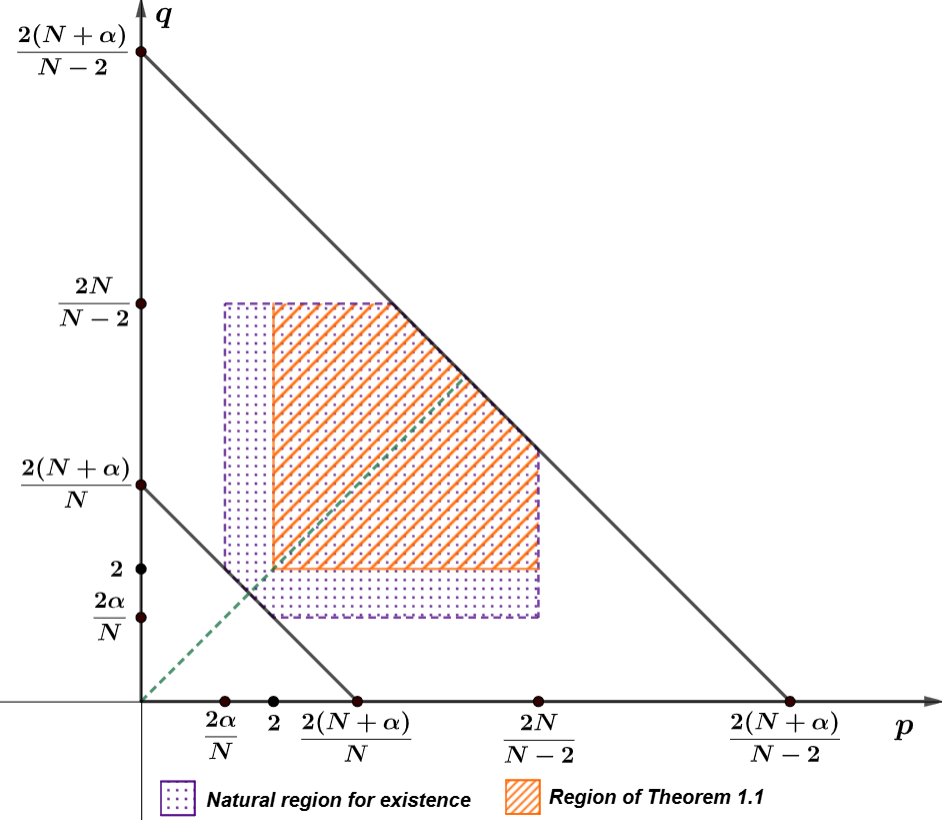}   
\end{center}
\caption{Region on the $(p,q)$-plane corresponding to Theorem \ref{thm:symmetry}. Here $N=3$ and $\alpha=1.9$.}\label{pic1}
\end{figure}

The Chen--Li--Ou approach was similarly employed in \cite[Theorem 2]{maClassificationPositiveSolitary2010} to obtain a symmetry result for positive solutions to \eqref{eqn:Choquard-type} under additional and fundamental integrability conditions on the considered solutions. Such type of conditions are not mentioned in Theorem \ref{thm:symmetry} because, as shown ahead at Section \ref{sect:ma-zhao}, thanks to \cite[Proof of Theorem 1.4]{boerStandingWavesNonlinear2025}, they are always satisfied by solutions in $E \parens{\real^N}$.

In the context of Choquard type equations, the values of parameters for which the moving plane method is applicable is a matter of current interest; see \cite[Section 3.3.3]{morozGuideChoquardEquation2017}. Theorem \ref{thm:symmetry} also applies to the Choquard equation \eqref{eqn:Choquard-type}, so it effectively extends \cite[Theorem 2]{maClassificationPositiveSolitary2010} to the context of systems and we hope that it helps to clarify the domain of application of the Chen--Li--Ou method for this class of problems.

We remark that \eqref{H2} is obtained from \eqref{H1} by further considering the technical restriction $p, q \geq 2$. This restriction ensures positive powers in the bounds established in Lemma \ref{lem:estimates} and it also appears in other papers that used the Chen--Li--Ou method, e.g. in \cite[Theorem 2]{maClassificationPositiveSolitary2010}.

System \eqref{eqn:Hartree-system:2}, with $\tau, \eta$ that are possibly different, naturally appears when considering the normalized system
\[
\begin{cases}
{
	\displaystyle
	- \Delta u + \tau u
	=
	\frac{2 p}{p + q}
	\parens*{I_\alpha \ast |v|^q}
	\abs{u}^{p - 2} u
}
&\text{in} ~ \real^N,
\\
\\
{
	\displaystyle
	- \Delta v + \eta v
	=
	\frac{2 q}{p + q}
	\parens*{I_\alpha \ast |u|^p}
	\abs{v}^{q - 2} v
}
&\text{in} ~ \real^N,
\\
\\
\norm{u}_{L^2}^2 = a;
\quad
\norm{v}_{L^2}^2 = b,
\end{cases}
\]
the unknowns being the Lagrange multipliers
$\tau, \eta \in \real$ and the functions
$u, v \colon \real^N \to \real$.
In particular, Theorem \ref{thm:symmetry} is applicable to pairs
$\parens{u, v} \in E \parens{\real^N}$
that solve the normalized system with Lagrange multipliers
$\tau, \eta > 0$.

We mention that Chen-Li-Ou type arguments are, to the best of our knowledge, still unavailable for \linebreak sublinear problems. Such an extension would  further enhance \cite[Theorem 2]{maClassificationPositiveSolitary2010} and \cite[Theorem 1]{chenClassificationSolutionsSystem2005} by removing the superlinearity assumptions and $p,q \geq 2$ in Theorem \ref{thm:symmetry}.

\subsection{Classification of positive ground states}
\label{intro:classification}

We turn our attention to positive ground states of the system obtained from \eqref{eqn:Hartree-system:2} in the case $p = q$ and $\tau = \eta$, that is,
\begin{equation}
\label{eqn:Choquard-system}
\begin{cases}
- \Delta u + \tau u
=
\parens*{I_\alpha \ast \abs{v}^p}
\abs{u}^{p - 2} u
&\text{in} ~ \real^N,
\\
- \Delta v + \tau v
=
\parens*{I_\alpha \ast \abs{u}^p}
\abs{v}^{p - 2} v
&\text{in} ~ \real^N.
\end{cases}
\end{equation}

The very basic question when investigating \eqref{eqn:Choquard-system} is that if it is indeed a system. In other words, are there positive solutions to \eqref{eqn:Choquard-system} with $u\neq v$\,? We solve this question in the framework of ground state solutions.

Before explaining our notion of ground state, we need to recall the variational characterization of solutions in
$E \parens{\real^N}$ to \eqref{eqn:Choquard-system}. Notice that we postponed this discussion up until now because the technique used to prove Theorem \ref{thm:symmetry} is completely independent from this variational characterization.

To start, notice that \eqref{H1} with $p=q$ is equivalent to $1 + \frac{\alpha}{N} < p < \frac{2_\alpha^*}{2}$.  Moreover, if $1 + \frac{\alpha}{N} < p < \frac{2_\alpha^*}{2}$, then there is a well-defined \emph{action functional}
$J \colon E \parens{\real^N} \to \real$
given by
\begin{equation}
\label{eqn:action-functional}
J \parens{u, v}
:=
\frac{1}{2}
\int_{\real^N}
	\abs{\nabla u}^2 + \abs{\nabla v}^2 + \tau \parens{u^2 + v^2}
\dif \mu
-
\frac{1}{p}
\int_{\real^N}
	\parens*{I_\alpha \ast \abs{u}^p} \abs{v}^q
\dif \mu.
\end{equation}
Usual arguments show that $J$ is a functional of class $C^1$. Furthermore, $\parens{u, v} \in E \parens{\real^N}$ is a solution to \eqref{eqn:Choquard-system} if, and only if,
$J' \parens{u, v} = 0$. In this context, we say that
$\parens{u, v}$ is a \emph{ground state} of \eqref{eqn:Choquard-system} when $\parens{u, v}$ solves the following minimization problem:
\begin{equation}
\label{eqn:min-N}
\begin{cases}
J \parens{u, v}
=
c
:=
\inf_{
	\parens{\widetilde{u}, \widetilde{v}}
	\in
	\N_J
}
J \parens{\widetilde{u}, \widetilde{v}};
\\
\parens{u, v} \in \N_J,
\end{cases}
\end{equation}
where
\[
\N_J
:=
\set*{
	\parens{u, v}
	\in
	E \parens{\real^N} \setminus \set*{\parens{0, 0}}:
	\langle J' \parens{u, v}, (u, v)\rangle = 0
}
\]
is the \emph{Nehari manifold} associated with $J$; see \cite[Definition 2.5 and Section 3]{boerStandingWavesNonlinear2025}.

Similarly, with $1 + \frac{\alpha}{N} < p < \frac{2_\alpha^*}{2}$, the Choquard equation \eqref{eqn:Choquard-type} is naturally associated with the \emph{action functional} 
defined as $\A \colon H^1 \parens{\real^N} \to \real$
\[
\A \parens{u}
=
\frac{1}{2}
\int_{\real^N}
	\abs{\nabla u}^2
	+
	\tau u^2
\dif \mu
-
\frac{1}{2 p}
\int_{\real^N}
	\parens*{I_\alpha \ast \abs{u}^p} \abs{u}^p
\dif \mu
\]
and we understand a \emph{ground state} of \eqref{eqn:Choquard-type} to be a solution to the following minimization problem:
\begin{equation}
\label{eqn:min-N_A}
\begin{cases}
\A \parens{u}
=
\T
:=
\inf_{v \in \N_A} \A \parens{v};
\\
u \in \N_\A,
\end{cases}
\end{equation}
where
\[
\N_\A
:=
\set*{
	u
	\in
	H^1 \parens{\real^N} \setminus \set{0}:
	\langle \A' \parens{u}, u\rangle = 0
}
\]
is the \emph{Nehari manifold} associated with
$\A$.

Our last main result is that we can classify the positive ground states of \eqref{eqn:Choquard-system} in function of the positive ground states of \eqref{eqn:Choquard-type}. This result is proved in Section \ref{classification} with arguments inspired by those in the proof of \cite[Theorem 1.4 (ii)]{wangStandingWavesCoupled2017}.
\begin{thm}
\label{thm:classification}
Suppose that $\tau > 0$, $N \in \nat$, $0 < \alpha < N$ and
$1 + \frac{\alpha}{N} < p < \frac{2^*_\alpha}{2}$.
Then $\parens{u, v}$ is a positive ground state solution of \eqref{eqn:Choquard-system} if, and only if, $u = v$ is a positive ground state solution of \eqref{eqn:Choquard-type}. 
\end{thm}

Regarding Theorem \ref{thm:classification} and a previous discussion in the section, we pose the following natural problem.

\medbreak
\noindent\textbf{Open problem 1.} Under the conditions in Theorem \ref{thm:classification}, all positive solutions of \eqref{eqn:Choquard-system} are of the form $(u,u)$ where $u$ is a positive solution of \eqref{eqn:Choquard-type}.
\medbreak

An indication that such result is expected is that a similar result is known to hold for a related weakly coupled cooperative system in the case $p = q = 2$; see \cite[Theorem 1.4]{wangStandingWavesCoupled2017}. Furthermore, a classification of positive solutions up to rescalings was already obtained in the case $\tau = \eta = 0$ in the Harly-Littlewood-Sobolev critical case $p=q=\frac{2_{\alpha}^*}{2}$; see \cite[Theorem 1.5]{wangClassificationQualitativeAnalysis2024}. Regarding Open problem 1, one must observe that $(u,-u)$ solves \eqref{eqn:Choquard-system} provided $u$ solves \eqref{eqn:Choquard-type}. This observation points to a broader open problem.

\medbreak
\noindent\textbf{Open problem 2.} Suppose the conditions in Theorem \ref{thm:classification}. Then $(u,v)$ solves \eqref{eqn:Choquard-system} if, and only if $v= \pm u$, where $u$ solves \eqref{eqn:Choquard-type}.
\medbreak
 
To finish this introduction we mention that throughout this paper,  given $\gamma \in [1, \infty]$, we let $\gamma' \in [1, \infty]$ denote its Hölder conjugate exponent.

\section{Symmetry of positive solutions}
\label{symmetry}

\subsection{The Hardy--Littlewood--Sobolev inequality}

The Hardy--Littlewood--Sobolev inequality plays a central role in our arguments, so we begin by recalling it as stated in \cite[Theorem 4.3]{liebAnalysis2001}.

\begin{prop}
\label{prop:HLS}
Suppose that $N \in \nat$ and $0 < \alpha < N$.
\begin{enumerate}
\item
If $\theta_1, \theta_2 \in \oci{1, \infty}$ are such that
$
\frac{1}{\theta_1} + \frac{1}{\theta_2}
=
\frac{N + \alpha}{N}
$,
then there exists
$C_{N, \alpha, \theta_1} > 0$
such that
\[
\abs*{
	\int_{\real^N} \int_{\real^N}
		\frac{f \parens{x} g \parens{y}}
			{\abs{x - y}^{N - \alpha}}
	\dif x \dif y
}
\leq
C_{N, \alpha, \theta_1}
\norm{f}_{L^{\theta_1}}
\norm{g}_{L^{\theta_2}}
\]
for every
$f \in L^{\theta_1} \parens{\real^N}$
and
$g \in L^{\theta_2} \parens{\real^N}$.
\item
Given
$r \in \ooi{1, \frac{N}{\alpha}}$,
there exists $C_{N, \alpha, r} > 0$ such that
\[
\int_{\real^N}
	\abs{I_\alpha \ast u}^
		{\frac{N r}{N - \alpha r}}
\dif \mu
\leq
C_{N, \alpha, r}
\parens*{
	\int_{\real^N} \abs{u}^r \dif \mu
}^{\frac{N}{N - \alpha r}}
\]
for every $u \in L^r \parens{\real^N}$.
\end{enumerate}
\end{prop}

\subsection{The Bessel potential}
\label{Bessel}

Our present goal is to recall a systematic procedure to obtain a strong solution to the equation
\begin{equation}
\label{potentials:tau}
- \Delta u + \tau u = f \quad \text{in} \quad \real^N,
\end{equation}
where $\tau > 0$, $f \in L^r \parens{\real^N}$ and
$1 < r < \infty$. As in \cite[Chapter V, Section 3]{steinSingularIntegralsDifferentiability2016} (see also \cite[Section 6.1.2]{grafakosModernFourierAnalysis2014}), we define the \emph{Bessel potential}
$
G_2
\colon
\real^N \setminus \set{0} \to \ooi{0, \infty}
$
as
\[
G_2 (x)
=
\frac{1}{4 \pi}
\int_0^\infty
	s^{-\frac{N}{2}}
	\exp\left(- \frac{\pi \abs{x}^2}{s}
	-
	\frac{s}{4 \pi}\right)
\dif s.
\]
Given $\tau > 0$, we define the rescaling of $G_2$ given by
\[
G_2^\tau \parens{x}
:=
\tau^{\frac{N - 2}{2}}
G_2 \parens*{\sqrt{\tau} x}.
\]

Convolution with $G_2^\tau$ returns functions with increased regularity and gives a representation formula for the solution of \eqref{potentials:tau} as stated more precisely in the next result.
\begin{lem}
[{see \cite[Chapter V, Section 3.3]{steinSingularIntegralsDifferentiability2016}}]
\label{lem:G_2-embedding}
Suppose that $N \in \nat$, $\tau > 0$ and $1 < r < \infty$. Then
\[
L^r \parens{\real^N} \ni f
\mapsto
G_2^\tau \ast f \in W^{2, r} \parens{\real^N}
\]
is a continuous linear operator. Moreover, $u := G_2^\tau \ast f \in W^{2, r} \parens{\real^N}$
is the strong solution to \eqref{potentials:tau}.

\end{lem}

\subsection{Notation for the moving plane method}

Let us fix the notation to be used in the moving plane method. For $\lambda \in \real$, set
\[
\Sigma_\lambda
=
\set*{
	\parens{x_1, \ldots, x_N} \in \real^N
	:
	x_1 \geq \lambda
}
\]
and for $x \in \real^N$, consider
\[
x^\lambda
:=
\parens{2 \lambda - x_1, x_2, \ldots, x_N}
\in
\real^N,
\]
that is, $x^\lambda$ denotes the point obtained by reflection with respect to
\[
\partial \Sigma_\lambda
=
\set*{
	\parens{\lambda, y_2, \ldots, y_N}
	:
	y_2, \ldots, y_N \in \real
}.
\]
Given
$u \colon \real^N \to \real$, let
$u_\lambda \colon \real^N \to \real$
be given by
$u_\lambda \parens{x} = u(x^\lambda)$
and set
\[
\Sigma^u_\lambda
=
\set*{
	x \in \Sigma_\lambda:
	u_\lambda \parens{x} > u \parens{x}
}.
\]
We finish the section with a simple but important remark, which follows the Dominated Convergence Theorem.

\begin{lem}
\label{lem:limit-of-integrals}
If $1 \leq s < \infty$ and
$u \in L^s \parens{\real^N}$, then
$
\norm{u_\lambda}_{L^s \parens{\Sigma_\lambda}} 
\to
0
$
as $\lambda \to - \infty$.
\end{lem}

\subsection{Technical preliminaries}

Throughout this section, we suppose that the hypotheses of Theorem \ref{thm:symmetry} are satisfied. In particular,
$\parens{u, v} \in E \parens{\real^N}$
always denotes a positive solution to \eqref{eqn:Hartree-system:2}. Furthermore, we set
\[
\phi_v = I_\alpha \ast v^q
\quad \text{and} \quad
\psi_u = I_\alpha \ast u^p
\]
to simplify the notation. Similarly as in \cite[p. 501]{vairaExistenceBoundStates2013} or \cite[Theorem 1.1]{wangClassificationQualitativeAnalysis2024},
our first goal is to show that $\parens{u, v}$ also solves an integral system.

\begin{lem}
\label{lem:integral-system}
It holds that
\begin{equation}
\label{lem:integral-system:1}
u \parens{x}
=
\frac{2 p}{p + q}
G_2^\tau \ast \parens{\phi_v u^{p - 1}} \parens{x}
=
\frac{2 p}{p + q}
\int_{\real^N}
	G_2^\tau \parens{x - y}
	\phi_v \parens{y} 
	u \parens{y}^{p - 1}
\dif y
\end{equation}
and
\begin{equation}
\label{lem:integral-system:2}
v \parens{x}
=
\frac{2 q}{p + q}
G_2^\eta \ast \parens{\psi_u v^{q - 1}} \parens{x}
=
\frac{2 q}{p + q}
\int_{\real^N}
	G_2^\eta \parens{x - y}
	\psi_u \parens{y} 
	v \parens{y}^{q - 1}
\dif y
\end{equation}
for every $x \in \real^N$.
\end{lem}
\begin{proof}
The identities follow from the same arguments, so we only prove the first one. Suppose that
$\xi \in C_c^\infty \parens{\real^N}$
and set $\Gamma = G_2^\tau \ast \xi$.
In view of the discussion in Section \ref{Bessel},
\[
\int_{\real^N} u \xi \dif \mu
=
\int_{\real^N}
	u
	\parens{- \Delta \Gamma + \tau \Gamma}
\dif \mu.
\]
The pair $\parens{u, v}$ solves \eqref{eqn:Hartree-system:2}, so
\[
\int_{\real^N} u \xi \dif \mu
=
\int_{\real^N}
	u
	\parens{- \Delta \Gamma + \tau\Gamma}
\dif \mu
=
\frac{2 p}{p + q}
\int_{\real^N}
	\phi_v u^{p - 1} \Gamma
\dif \mu.
\]
The function $G_2^\tau$ is even, so a change of variable shows that
\begin{align*}
\int_{\real^N} u \xi \dif \mu
&=
\frac{2 p}{p + q}
\int_{\real^N}
	\phi_v u^{p - 1} \Gamma
\dif \mu =
\frac{2 p}{p + q}
\int_{\real^N} \int_{\real^N}
	G_2^\tau \parens{y} \xi \parens{x - y}
	\phi_v \parens{x}
	u \parens{x}^{p - 1}
\dif x \dif y
\\
&=
\frac{2 p}{p + q}
\int_{\real^N}
	\parens*{
		G_2^\tau
		\ast
		\parens{\phi_v u^{p - 1}}
	}
	\xi
\dif \mu.
\end{align*}
Finally, the result follows from the fact that
$\xi$ denotes an arbitrary function in $C_c^\infty(\real^N)$.
\end{proof}

Let us establish a few identities which are similar to those in \cite[(3.4) and (3.5)]{vairaExistenceBoundStates2013} or \cite[Lemmas 4.1 and 4.2]{wangClassificationQualitativeAnalysis2024}.

\begin{lem}
\label{lem:u_lambda-u}
The following identities are satisfied for every $x \in \real^N$:
\begin{multline}
\label{lem:u_lambda-u:1}
u_\lambda \parens{x} - u \parens{x}
=
\frac{2 p}{p + q}
\int_{\Sigma_\lambda}
	\parens*{
		G_2^\tau \parens{x - y}
		-
		G_2^\tau \parens{x^\lambda - y}
	}
	\parens*{
		\phi_{v, \lambda} \parens{y} 
		u_\lambda \parens{y}^{p - 1}
		-
		\phi_v \parens{y}
		u \parens{y}^{p - 1}
	}
\dif y \, ,
\end{multline}
\begin{multline}
\label{lem:u_lambda-u:2}
v_\lambda \parens{x} - v \parens{x}
=
\frac{2 q}{p + q}
\int_{\Sigma_\lambda}
	\parens*{
		G_2^\eta \parens{x - y}
		-
		G_2^\eta \parens{x^\lambda - y}
	}
	\parens*{
		\psi_{u, \lambda} \parens{y} 
		v_\lambda \parens{y}^{q - 1}
		-
		\psi_u \parens{y}
		v \parens{y}^{q - 1}
	}
\dif y \, ,
\end{multline}
\begin{equation}
\label{lem:u_lambda-u:3}
\phi_{v, \lambda} \parens{x}
-
\phi_v \parens{x}
=
\int_{\Sigma_\lambda}
	\parens*{
		I_\alpha \parens{x - y}
		-
		I_\alpha \parens{x^\lambda - y}
	}
	\parens*{
		v_\lambda \parens{y}^q
		-
		v \parens{y}^q
	}
\dif y\, ,
\end{equation}
and
\begin{equation}
\label{lem:u_lambda-u:4}
\psi_{u, \lambda} \parens{x}
-
\psi_u \parens{x}
=
\int_{\Sigma_\lambda}
	\parens*{
		I_\alpha \parens{x - y}
		-
		I_\alpha \parens{x^\lambda - y}
	}
	\parens*{
		u_\lambda \parens{y}^p
		-
		u \parens{y}^p
	}
\dif y.
\end{equation}
\end{lem}
\begin{proof}
Once again, we only prove that the first one holds. It follows from Lemma \ref{lem:integral-system} that
\begin{equation}
\label{lem:u_lambda-u:5}
\begin{aligned}
u \parens{x}
&=
\frac{2 p}{p + q}
\int_{\Sigma_\lambda}
	G_2^\tau \parens{x - y}
	\phi_v \parens{y}
	u \parens{y}^{p - 1}
\dif y +
\frac{2 p}{p + q}
\int_{\real^N \setminus \Sigma_\lambda}
	G_2^\tau \parens{x - y}
	\phi_v \parens{y}
	u \parens{y}^{p - 1}
\dif y.
\end{aligned}
\end{equation}
Due to the fact that
\begin{equation}
\label{eqn:reflection}
\real^N \ni y
\mapsto
y^\lambda \in \real^N
\end{equation}
is idempotent, it follows that
\[
\int_{\real^N \setminus \Sigma_\lambda}
	G_2^\tau \parens{x - y}
	\phi_v \parens{y}
	u \parens{y}^{p - 1}
\dif y
=
\int_{\real^N \setminus \Sigma_\lambda}
	G_2^\tau \parens*{\parens{x^\lambda}^\lambda - y}
	\phi_v \parens{y}
	u \parens{y}^{p - 1}
\dif y.
\]
It is easy to verify that
$|x - y^\lambda| = |x^\lambda - y|$
for every
$y \in \real^N$. As such, the change of variable $z = y^\lambda$ shows that
\[
\int_{\real^N \setminus \Sigma_\lambda}
	G_2^\tau \parens{x - y}
	\phi_v \parens{y}
	u \parens{y}^{p - 1}
\dif y
=
\int_{\Sigma_\lambda}
	G_2^\tau \parens{x^\lambda - z}
	\phi_{v, \lambda} \parens{z}
	u_\lambda \parens{z}^{p - 1}
\dif z
\]
because $G_2^\tau$ is radial. Then, in view of \eqref{lem:u_lambda-u:5}, 
\begin{equation}
\label{lem:u_lambda-u:6}
u \parens{x}
=
\frac{2 p}{p + q}
\int_{\Sigma_\lambda}
	G_2^\tau \parens{x - y}
	\phi_v \parens{y}
	u \parens{y}^{p - 1}
	+
	G_2^\tau \parens{x^\lambda - y}
	\phi_{v, \lambda} \parens{y}
	u_\lambda \parens{y}^{p - 1}
\dif y.
\end{equation}
Using again the fact that \eqref{eqn:reflection}
is idempotent, if follows that
\begin{equation}
\label{lem:u_lambda-u:7}
\begin{aligned}
u_\lambda \parens{x}
=
u \parens{x^\lambda}
&=
\frac{2 p}{p + q}
\int_{\Sigma_\lambda}
	G_2^\tau \parens{x^\lambda - y}
	\phi_v \parens{y}
	u \parens{y}^{p - 1}
\dif y +
\frac{2 p}{p + q}
\int_{\Sigma_\lambda}
	G_2^\tau \parens{x - y}
	\phi_{v, \lambda} \parens{y}
	u_\lambda \parens{y}^{p - 1}
\dif y.
\end{aligned}
\end{equation}
Finally, \eqref{lem:u_lambda-u:1} follows from \eqref{lem:u_lambda-u:6} and \eqref{lem:u_lambda-u:7}.
\end{proof}

Now, we obtain inequalities which are similar to those developed at Step 1 in \cite[p. 502--503]{vairaExistenceBoundStates2013}.

\begin{lem}
\label{lem:preliminary-inequalities}
Suppose that $\lambda \leq 0$. The following inequalities are satisfied:
\begin{equation}
\label{lem:preliminary-inequalities:1}
0
<
G_2^\tau \parens{x - y}
-
G_2^\tau \parens{x^\lambda - y}
<
G_2^\tau \parens{x - y}
\end{equation}
and
\begin{equation}
\label{lem:preliminary-inequalities:2}
0
<
I_\alpha \parens{x - y}
-
I_\alpha \parens{x^\lambda - y}
<
I_\alpha \parens{x - y}
\end{equation}
for every $\tau > 0$, $x \in \interior \Sigma_\lambda$
and
$y \in \interior \Sigma_\lambda \setminus \set{x}$.
The inequalities that follow are also satisfied:
\begin{equation}
\label{lem:preliminary-inequalities:3}
\begin{aligned}
0
\leq
u_\lambda \parens{x} - u \parens{x}
&\leq
\frac{2 p \parens{p - 1}}{p + q}
\int_{\Sigma_\lambda^u}
	G_2^\tau \parens{x - y}
	\phi_{v, \lambda} \parens{y}
	u_\lambda \parens{y}^{p - 2}
	\parens*{
		u_\lambda \parens{y}
		-
		u \parens{y}
	}
\dif y
\\
&+
\frac{4 p}{p + q}
\int_{\Sigma_\lambda^{\phi_v}}
	G_2^\tau \parens{x - y}
	u \parens{y}^{p - 1}
	\parens*{
		\phi_{v, \lambda} \parens{y}
		-
		\phi_v \parens{y}
	}
\dif y
\end{aligned}
\end{equation}
for every $x \in \Sigma_\lambda^u$;
\begin{equation}
\label{lem:preliminary-inequalities:5}
\begin{aligned}
0
\leq
v_\lambda \parens{x} - v \parens{x}
&\leq
\frac{2 q \parens{q - 1}}{p + q}
\int_{\Sigma_\lambda^v}
	G_2^\eta \parens{x - y}
	\psi_{u, \lambda} \parens{y}
	v_\lambda \parens{y}^{q - 2}
	\parens*{
		v_\lambda \parens{y}
		-
		v \parens{y}
	}
\dif y
\\
&+
\frac{4 q}{p + q}
\int_{\Sigma_\lambda^{\phi_v}}
	G_2^\eta \parens{x - y}
	v \parens{y}^{q - 1}
	\parens*{
		\psi_{u, \lambda} \parens{y}
		-
		\psi_u \parens{y}
	}
\dif y
\end{aligned}
\end{equation}
for every
$x \in \Sigma_\lambda^v$;
\begin{equation}
\label{lem:preliminary-inequalities:7}
0
\leq
\phi_{v, \lambda} \parens{x} - \phi_v \parens{x}
\leq
q
\int_{\Sigma_\lambda^v}
	I_\alpha \parens{x - y}
	v_\lambda \parens{y}^{q - 1}
	\parens*{v_\lambda \parens{y} - v \parens{y}}
\dif y
\end{equation}
for every $x \in \Sigma_\lambda^{\phi_v}$ and
\begin{equation}
\label{lem:preliminary-inequalities:8}
0
\leq
\psi_{u, \lambda} \parens{x} - \psi_u \parens{x}
\leq
p
\int_{\Sigma_\lambda^u}
	I_\alpha \parens{x - y}
	u_\lambda \parens{y}^{p - 1}
	\parens*{u_\lambda \parens{y} - u \parens{y}}
\dif y
\end{equation}
for every
$x \in \Sigma_\lambda^{\psi_u}$.
\end{lem}
\begin{proof} Proof of \eqref{lem:preliminary-inequalities:1} and \eqref{lem:preliminary-inequalities:2}.
The positive functions $G_2^\tau$ and $I_\alpha$ are radially decreasing and it is easy to verify that
$\abs{x - y} \leq \abs{x^\lambda - y}$
for every
$x, y \in \Sigma_\lambda$, 
so
\[
0
\leq
G_2^\tau \parens{x - y}
-
G_2^\tau \parens{x^\lambda - y}
<
G_2^\tau \parens{x - y}
\]
and
\[
0
\leq
I_\alpha \parens{x - y}
-
I_\alpha \parens{x^\lambda - y}
<
I_\alpha \parens{x - y}.
\]
On one hand, the equalities
\[
0
=
G_2 \parens{x - y}
-
G_2 \parens{x^\lambda - y}
=
I_\alpha \parens{x - y}
-
I_\alpha \parens{x^\lambda - y}
\]
only hold when
$\abs{x - y} = \abs{x^\lambda - y} > 0$.
On the other hand, the equality
$\abs{x - y} = \abs{x^\lambda - y}$
is satisfied precisely when $x_1 = \lambda$ or
$y_1 = \lambda$, hence the result.

\paragraph{Proof of \eqref{lem:preliminary-inequalities:3} and \eqref{lem:preliminary-inequalities:5}.}
We only prove that \eqref{lem:preliminary-inequalities:3} holds. In view of \eqref{lem:u_lambda-u:1},
\[
\begin{aligned}
u_\lambda \parens{x} - u \parens{x}
&\leq
\frac{2 p}{p + q}
\int_{\Sigma_\lambda^u}
	\parens*{
		G_2^\tau \parens{x - y}
		-
		G_2^\tau \parens{x^\lambda - y}
	}
	\parens*{
		\phi_{v, \lambda} \parens{y}
		u_\lambda \parens{y}^{p - 1}
		-
		\phi_v \parens{y}
		u \parens{y}^{p - 1}
	}
\dif y
\\
&+
\frac{2 p}{p + q}
\int_{\Sigma_\lambda \setminus \Sigma_\lambda^u}
	\parens*{
		G_2^\tau \parens{x - y}
		-
		G_2^\tau \parens{x^\lambda - y}
	}
	u \parens{y}^{p - 1}
	\parens*{
		\phi_{v, \lambda} \parens{y}
		-
		\phi_v \parens{y}
	}
\dif y;
\\
&\leq
\frac{2 p}{p + q}
\int_{\Sigma_\lambda^u}
	\parens*{
		G_2^\tau \parens{x - y}
		-
		G_2^\tau \parens{x^\lambda - y}
	}
	\parens*{
		\phi_{v, \lambda} \parens{y} u_\lambda \parens{y}^{p - 1}
		-
		\phi_v \parens{y} u \parens{y}^{p - 1}
	}
\dif y
\\
&+
\frac{2 p}{p + q}
\int_{(\Sigma_\lambda \setminus \Sigma_\lambda^u) \cap \Sigma_\lambda^{\phi_v}}
	\parens*{
		G_2^\tau \parens{x - y}
		-
		G_2^\tau \parens{x^\lambda - y}
	}
	u \parens{y}^{p - 1}
	\parens*{
		\phi_{v, \lambda} \parens{y}
		-
		\phi_v \parens{y}
	}
\dif y;
\\
&\leq
\frac{2 p}{p + q}
\int_{\Sigma_\lambda^u}
	\parens*{
		G_2^\tau \parens{x - y}
		-
		G_2^\tau \parens{x^\lambda - y}
	}
	\parens*{
		\phi_{v, \lambda} \parens{y}
		u_\lambda \parens{y}^{p - 1}
		-
		\phi_v \parens{y}
		u \parens{y}^{p - 1}
	}
\dif y
\\
&+
\frac{2 p}{p + q}
\int_{\Sigma_\lambda^{\phi_v}}
	\parens*{
		G_2^\tau \parens{x - y}
		-
		G_2^\tau \parens{x^\lambda - y}
	}
	u \parens{y}^{p - 1}
	\parens*{
		\phi_{v, \lambda} \parens{y}
		-
		\phi_v \parens{y}
	}
\dif y;
\\
&\leq
\frac{2 p}{p + q}
\int_{\Sigma_\lambda^u}
	\parens*{
		G_2^\tau \parens{x - y}
		-
		G_2^\tau \parens{x^\lambda - y}
	}
	\phi_v \parens{y}
	\parens*{
		u_\lambda \parens{y}^{p - 1}
		-
		u \parens{y}^{p - 1}
	}
\dif y
\\
&+
\frac{4 p}{p + q}
\int_{\Sigma_\lambda^{\phi_v}}
	\parens*{
		G_2^\tau \parens{x - y}
		-
		G_2^\tau \parens{x^\lambda - y}
	}
	u_\lambda \parens{y}^{p - 1}
	\parens*{
		\phi_{v, \lambda} \parens{y}
		-
		\phi_v \parens{y}
	}
\dif y.
\end{aligned}
\]
Due to \eqref{lem:preliminary-inequalities:1},
\[
\begin{aligned}
u_\lambda \parens{x} - u \parens{x}
&\leq
\frac{2 p}{p + q}
\int_{\Sigma_\lambda^u}
	G_2^\tau \parens{x - y}
	\phi_v \parens{y}
	\parens*{
		u_\lambda \parens{y}^{p - 1}
		-
		u \parens{y}^{p - 1}
	}
\dif y
\\
&+
\frac{4 p}{p + q}
\int_{\Sigma_\lambda^{\phi_v}}
	G_2^\tau \parens{x - y}
	u_\lambda \parens{y}^{p - 1}
	\parens*{
		\phi_{v, \lambda} \parens{y}
		-
		\phi_v \parens{y}
	}
\dif y.
\end{aligned}
\]
To estimate the first term on the right-hand side, by the mean value inequality and the definition of $\Sigma_\lambda^u$,
\[
0
\leq
u_\lambda^{p - 1} - u^{p - 1}
\leq
\parens{p - 1} u_\lambda^{p - 2}
\parens{u_\lambda - u}
\quad \text{in} \quad
\Sigma_\lambda^u.
\]
The result follows from this estimate.

\paragraph{Proof of \eqref{lem:preliminary-inequalities:7}, \eqref{lem:preliminary-inequalities:8}.}
We only prove that \eqref{lem:preliminary-inequalities:7} is satisfied. As in the previous proof,
\[
0
\leq
v_\lambda^q - v^q
\leq
q v_\lambda^{q - 1} \parens{v_\lambda - v}
\quad \text{in} \quad
\Sigma_\lambda^v.
\]
In view of this inequality, the result is a corollary of \eqref{lem:u_lambda-u:3} and \eqref{lem:preliminary-inequalities:2}.
\end{proof}

\subsection{Some integrability conditions}
\label{sect:ma-zhao}

Instead of directly proving Theorem \ref{thm:symmetry}, we will consider a slightly different result with a condition inspired by the hypotheses in \cite[Theorem 2]{maClassificationPositiveSolitary2010}; see Theorem \ref{thm:most-general} ahead. Afterwards, we prove in Lemma \ref{lem:ma-zhao} that this condition is automatically satisfied.

\begin{cond}
\label{cond:ma-zhao}
There exist
\[
\theta_1, \theta_2 \in \ooi*{1, \frac{N}{\alpha}}\, ,
\quad
r, h \in \ooi{1, \infty}\, ,
\quad
\beta_0, \beta_1, \gamma_0, \gamma_1
\in
\ooi*{\frac{N}{N - \alpha}, \infty}
\]
and
\[s_0, s_1, s_2, s_3, t_0, t_1, t_2, t_3 \in \coi{2, \infty}\]
such that
\begin{center}
\begin{tabular}{l l}
$
\begin{cases}
u
\in
W^{2, r} \parens{\real^N} \hookrightarrow L^{s_0} \parens{\real^N},
\\
u
\in
L^{s_1} \parens{\real^N}
\cap
L^{s_2} \parens{\real^N}
\cap
L^{s_3} \parens{\real^N},
\\
v
\in
W^{2, h} \parens{\real^N} \hookrightarrow L^{t_0} \parens{\real^N},
\\
v
\in
L^{t_1} \parens{\real^N}
\cap
L^{t_2} \parens{\real^N}
\cap
L^{t_3} \parens{\real^N},
\end{cases}
$
&
$
\begin{cases}
I_\alpha \ast u^p \in L^{\gamma_0} \parens{\real^N} \cap L^{\gamma_1} \parens{\real^N},
\\
I_\alpha \ast v^q \in L^{\beta_0} \parens{\real^N} \cap L^{\beta_1} \parens{\real^N},
\end{cases}
$
\\
\\
$
\begin{cases}
s_1 \geq p - 2, ~
s_2 \geq p - 1, ~
s_3 \geq p - 1,
\\
\frac{1}{\beta_1}
+
\frac{p - 2}{s_1}
+
\frac{1}{s_0}
=
\frac{1}{r},
\\
\frac{p - 1}{s_2}
+
\frac{1}{\beta_0}
=
\frac{1}{r},
\\
\frac{1}{\beta_0} + \frac{\alpha}{N}
=
\frac{1}{\theta_2},
\\
\frac{q - 1}{t_3} + \frac{1}{t_0}
=
\frac{1}{\theta_2}\, ,
\\
\end{cases}
$
and
&
$
\begin{cases}
t_1 \geq q - 2, ~
t_2 \geq q - 1, ~
t_3 \geq q - 1,
\\
\frac{1}{\gamma_1}
+
\frac{q - 2}{t_1}
+
\frac{1}{t_0}
=
\frac{1}{h},
\\
\frac{q - 1}{t_2}
+
\frac{1}{\gamma_0}
=
\frac{1}{h},
\\
\frac{1}{\gamma_0} + \frac{\alpha}{N}
=
\frac{1}{\theta_1},
\\
\frac{p - 1}{s_3} + \frac{1}{s_0}
=
\frac{1}{\theta_1}.
\end{cases}
$
\end{tabular}
\end{center}
\end{cond}

The motivation for Condition \ref{cond:ma-zhao} will become clearer in Lemma \ref{lem:estimates}, where we show that it collects sufficient integrability hypotheses under which we can appropriately use the Hardy--Littlewood--Sobolev and Hölder inequalities to obtain adequate estimates for the Chen--Li--Ou method.

\begin{thm}
\label{thm:most-general}
Suppose that $\tau, \eta > 0$, $N \in \nat$, $0 < \alpha < N$, \eqref{H2} holds, $\parens{u, v} \in E \parens{\real^N}$ is a positive solution to \eqref{eqn:Hartree-system:2} and Condition \ref{cond:ma-zhao} is satisfied. Then there exists
$x_0 \in  \real^N$ such that $u (\cdot - x_0)$
and $v (\cdot - x_0)$ are radial and strictly radially decreasing.
\end{thm}

In order to deduce that Theorem \ref{thm:most-general} implies Theorem \ref{thm:symmetry}, we only have to show that Condition \ref{cond:ma-zhao} is automatically satisfied for any  solution $\parens{u, v} \in E \parens{\real^N}$ of \eqref{eqn:Hartree-system:2}. This is done in the following lemma by arguing as in \cite[Proof of Theorem 1.4]{boerStandingWavesNonlinear2025}.

\begin{lem}
\label{lem:ma-zhao}
Suppose that $\tau, \eta > 0$, $N \in \nat$, $0 < \alpha < N$, \eqref{H2} holds and $\parens{u, v} \in E \parens{\real^N}$ is a solution to \eqref{eqn:Hartree-system:2}. Then Condition \ref{cond:ma-zhao} is satisfied.
\end{lem}
\begin{proof}
Hypothesis \eqref{H2} implies \eqref{H1}, so we can fix
$\theta_1, \theta_2 \in \ooi{1, \frac{N}{\alpha}}$
such that \eqref{eqn:thetas} holds. Let
\begin{multicols}{2}
\begin{itemize}
\item
$s_0 = s_1 = s_2 = s_3 = \theta_1 p \in \ooi{2, 2^*}$\,,
\item
$t_0 = t_1 = t_2 = t_3 = \theta_2 q \in \ooi{2, 2^*}$\,,
\item
$r = s_0' \in \ooi{\parens{2^*}', 2}$\,,
\item
$h = t_0' \in \ooi{\parens{2^*}', 2}$\,,
\item
$
\beta = \beta_1
=
\parens*{\frac{1}{\theta_2} - \frac{\alpha}{N}}^{- 1}
\in
\ooi*{\frac{N}{N - \alpha}, \infty}
$\, ,
\item
$
\gamma = \gamma_1
=
\parens*{\frac{1}{\theta_1} - \frac{\alpha}{N}}^{- 1}
\in
\ooi*{\frac{N}{N - \alpha}, \infty}
$\, .
\end{itemize}
\end{multicols}

There are two possible cases.
\subparagraph{Case $2 r < N$.}
This case occurs precisely when either (i) $N = 3$ and $s_0 > 3$ or (ii) $N \geq 4$. On one hand,
$s_0 < \frac{s_0'N}{N - 2 s_0'} = \frac{rN}{N - 2 r}$
because $s_0 < \frac{2 N}{N - 2}$. On the other hand,
$s_0 > s_0' = r$ because $s_0 > 2$. In particular, it follows from the Sobolev embeddings that
$
W^{2, r} \parens{\real^N} \hookrightarrow L^{s_0} \parens{\real^N}
$. 

\subparagraph{Case $2 r \geq N$.}
This case occurs precisely when either (i) $N \in \set{1, 2}$ or (ii) $N = 3$ and $2 < s_0 \leq 3$. Either way, we have
$1 < r = s_0' < s_0 < \infty$ because $2 < s_0 < \infty$. Once again, it follows from the Sobolev embeddings that
$
W^{2, r} \parens{\real^N} \hookrightarrow L^{s_0} \parens{\real^N}
$.
Indeed,
$W^{2, r} \parens{\real^N} \hookrightarrow L^s \parens{\real^N}$
for every $s \in \coi{r, \infty}$ as stated in \cite[Theorem 4.12]{adamsSobolevSpaces2003}. 

\

Similar arguments show that
$
W^{2, h} \parens{\real^N} \hookrightarrow L^{t_0} \parens{\real^N}
$
and the equalities in Condition \ref{cond:ma-zhao} may be verified with elementary computations, hence the result.
\end{proof}

We finish the section with an analog to \cite[(3.6) and (3.8)]{vairaExistenceBoundStates2013} or \cite[(4.25)--(4.27)]{wangClassificationQualitativeAnalysis2024}.

\begin{lem}
\label{lem:estimates}
Suppose that the conditions in Theorem \ref{thm:most-general} are satisfied. Then there exist positive constants $K_1, K_2, K_3, K_4$ such that
\begin{equation}
\label{eqn:estimate-phi-psi:1}
\norm{\phi_{v, \lambda} - \phi_v}_{L^{\beta_0} \parens{\Sigma_\lambda^{\phi_v}}}
\leq
K_1
\norm{v_\lambda}_{
	L^{t_3} \parens{\Sigma_\lambda^v}
}^{q - 1}
\norm{v_\lambda - v}_
	{L^{t_0} \parens{\Sigma_\lambda^v}};
\end{equation}
\begin{equation}
\label{eqn:estimate-phi-psi:2}
\norm{\psi_{u, \lambda} - \psi_u}_{L^{\gamma_0} \parens{\Sigma_\lambda^{\psi_u}}}
\leq
K_2
\norm{u_\lambda}_{L^{s_3} \parens{\Sigma_\lambda^u}}^{p - 1}
\norm{u_\lambda - u}_{L^{s_0} \parens{\Sigma_\lambda^u}};
\end{equation}
\begin{multline}
\label{eqn:estimate-u_lambda-u:1}
\norm{u_\lambda - u}_{L^{s_0} \parens{\Sigma_\lambda^u}}
\leq
\\
\leq
K_3
\parens*{
	\norm{\phi_{v, \lambda}}_{L^{\beta_1} \parens{\Sigma_\lambda^u}}
	\norm{u_\lambda}_{L^{s_1} \parens{\Sigma_\lambda^u}}^{p - 2}
	\norm{u_\lambda - u}_{L^{s_0} \parens{\Sigma_\lambda^u}}
	+
	\norm{u}_{L^{s_2}}^{p - 1}
	\norm{v_\lambda}_{L^{t_3} \parens{\Sigma_\lambda^v}}^{q - 1}
	\norm{v_\lambda - v}_{L^{t_0} \parens{\Sigma_\lambda^v}}
}
\end{multline}
and
\begin{multline}
\label{eqn:estimate-u_lambda-u:2}
\norm{v_\lambda - v}_{L^{t_0} \parens{\Sigma_\lambda^v}}
\leq
\\
\leq
K_4
\parens*{
	\norm{\psi_{u, \lambda}}_{L^{\gamma_1} \parens{\Sigma_\lambda^v}}
	\norm{v_\lambda}_{L^{t_1} \parens{\Sigma_\lambda^v}}^{q - 2}
	\norm{v_\lambda - v}_{L^{t_0} \parens{\Sigma_\lambda^v}}
	+
	\norm{v}_{L^{t_2}}^{q - 1}
	\norm{u_\lambda}_{L^{s_3} \parens{\Sigma_\lambda^u}}^{p - 1}
	\norm{u_\lambda - u}_{L^{s_0} \parens{\Sigma_\lambda^u}}
}
\end{multline}
for every $\lambda \leq 0$.
\end{lem}
\begin{proof}
The estimates follow from similar arguments, so we only show that \eqref{eqn:estimate-phi-psi:2}, \eqref{eqn:estimate-u_lambda-u:2} are satisfied.

\paragraph{Proof of \eqref{eqn:estimate-phi-psi:2}.}
Due to \eqref{lem:preliminary-inequalities:8},
\[
\norm{\psi_{u, \lambda} - \psi_u}_{L^{\gamma_0} \parens{\Sigma_\lambda^{\psi_u}}}
\leq
p
\norm*{
	\int_{\Sigma_\lambda^u}
		I_\alpha \parens{\cdot - y}
		u_\lambda \parens{y}^{p - 1}
		\parens*{
			u_\lambda \parens{y} - u \parens{y}
		}
	\dif y
}_{L^{\gamma_0} \parens{\Sigma_\lambda^{\psi_u}}}.
\]
In view of Proposition \ref{prop:HLS}, there exists a positive constant $K_2$ such that
\[
\norm*{
	\int_{\Sigma_\lambda^u}
		I_\alpha \parens{\cdot - y}
		u_\lambda \parens{y}^{p - 1}
		\parens*{
			u_\lambda \parens{y} - u \parens{y}
		}
	\dif y
}_{L^{\gamma_0} \parens{\Sigma_\lambda^{\psi_u}}}
\leq
K_2
\norm{u_\lambda}_
	{L^{s_3} \parens{\Sigma_\lambda^u}}^{p - 1}
\norm{u_\lambda - u}_
	{L^{s_0} \parens{\Sigma_\lambda^u}},
\]
hence the result.

\paragraph{Proof of \eqref{eqn:estimate-u_lambda-u:2}.}
Considering \eqref{lem:preliminary-inequalities:5}, it follows that
\[
\begin{aligned}
0
\leq
v_\lambda \parens{x} - v \parens{x}
&\leq
\frac{2 q (q - 1)}{p + q}
\int_{\Sigma_\lambda^v}
	G_2^\eta \parens{x - y}
	\psi_{u, \lambda} \parens{y}
	v_\lambda \parens{y}^{q - 2}
	\parens*{v_\lambda \parens{y} - v(y)}
\dif y
\\
&+
\frac{4 q}{p + q}
\int_{\Sigma_\lambda^{\psi_u}}
	G_2^\eta \parens{x - y}
	v \parens{y}^{q - 1}
	\parens*{
		\psi_{u, \lambda} \parens{y}
		-
		\psi_u \parens{y}
	}
\dif y.
\end{aligned}
\]
Due to Lemma \ref{lem:G_2-embedding}, there exists a positive constant $K_4'$ such that
\begin{align*}
\norm{v_\lambda - v}_{L^{t_0} \parens{\Sigma_\lambda^v}}
&\leq
K_4'
\norm*{
	\psi_{u, \lambda}
	v_\lambda^{q - 2}
	\parens{v_\lambda - v}
}_{L^h \parens{\Sigma_\lambda^v}}
+
K_4'
\norm*{
	v^{q - 1}
	\parens{\psi_{u, \lambda} - \psi_u}
}_{L^h \parens{\Sigma_\lambda^v}}
\\
&\leq
K_4'
\norm{\psi_{u, \lambda}}_{L^{\gamma_1} \parens{\Sigma_\lambda^v}}
\norm{v_\lambda}_{L^{t_1} \parens{\Sigma_\lambda^v}}^{q - 2}
\norm{v_\lambda - v}_{L^{t_0} \parens{\Sigma_\lambda^v}}
+
K_4	'
\norm{v}_{L^{t_2} \parens{\Sigma_\lambda^{\psi_u}}}^{q - 1}
\norm{\psi_{u, \lambda} - \psi_u}_{L^{\gamma_0} \parens{\Sigma_\lambda^{\psi_u}}}.
\end{align*}
At this point, the result follows from \eqref{eqn:estimate-phi-psi:2}.
\end{proof}

\begin{proof}[\textbf{Proof of Theorem \ref{thm:most-general}}]

\noindent {\bf Step 1:} \emph{Beginning of the moving plane method.} The first goal is to prove that
\begin{equation}
\label{proof:-1}
\text{
there exists
$M \geq 0$
such that
$u \geq u_\lambda, v \geq v_\lambda$
in
$\Sigma_\lambda$
for every
$\lambda \leq - M$.
}
\end{equation}
It suffices to show that
\begin{equation}
\label{proof:0}
\text{
there exists
$M \geq 0$
such that
$
\norm{u_\lambda - u}_
	{L^{s_0} \parens{\Sigma_\lambda^u}}
=
\norm{v_\lambda - v}_
	{L^{t_0} \parens{\Sigma_\lambda^v}}
=
0
$
for every
$\lambda \leq - M$.
}
\end{equation}
Indeed, $u \geq u_\lambda$ in
$\Sigma_\lambda \setminus \Sigma_\lambda^u$ and
$u < u_\lambda$ in $\Sigma_\lambda^u$ by definition. As such,
$
\norm{u_\lambda - u}_
	{L^{s_0} \parens{\Sigma_\lambda^u}}
=
0
$ implies
$\Sigma_\lambda^u = \emptyset$, hence the result.

\noindent{\bf Proof of \eqref{proof:0}.}
Let $K_3, K_4$ denote the positive constants furnished by Lemma \ref{lem:estimates}. Due to Lemma \ref{lem:limit-of-integrals}, we can fix $\lambda_1 \leq 0$ such that if
$\lambda \leq \lambda_1$, then
\begin{equation}
\label{proof:0.5}
K_4
\norm{\psi_{u, \lambda}}_
	{L^{\gamma_1} \parens{\Sigma_\lambda^v}}
\norm{v_\lambda}_
	{L^{t_1} \parens{\Sigma_\lambda^v}}^{q - 2}
<
\frac{1}{2}
\quad \text{and} \quad
K_4
\norm{v}_{L^{t_2}}^{q - 1}
\norm{u_\lambda}_{L^{s_3} \parens{\Sigma_\lambda^u}}^{p - 1}
<
\frac{1}{2},
\end{equation}
so it follows from \eqref{eqn:estimate-u_lambda-u:2} that
\begin{equation}
\label{proof:1}
\norm{v_\lambda - v}_{L^{t_0} \parens{\Sigma_\lambda^v}}
\leq
\norm{u_\lambda - u}_{L^{s_0} \parens{\Sigma_\lambda^u}}.
\end{equation}
Similarly, we can fix $\lambda_2 \leq \lambda_1$ such that if
$\lambda \leq \lambda_2$, then
\begin{equation}
\label{proof:1.5}
K_3
\norm{\phi_{v, \lambda}}_{L^{\beta_1} (\Sigma_\lambda^u)}
\norm{u_\lambda}_{L^{s_1} \parens{\Sigma_\lambda^u}}^{p - 2}
<
\frac{1}{4}
\quad \text{and} \quad
K_3
\norm{u}_{L^{s_2}}^{p - 1}
\norm{v_\lambda}_{L^{t_3} \parens{\Sigma_\lambda^v}}^{q - 1}
<
\frac{1}{4}.
\end{equation}
In view of \eqref{proof:1} and \eqref{eqn:estimate-u_lambda-u:1}, we obtain
\[
3 \norm{u_\lambda - u}_{L^{s_0} \parens{\Sigma_\lambda^u}}
\leq
\norm{v_\lambda - v}_{L^{t_0} \parens{\Sigma_\lambda^v}}
\leq
\norm{u_\lambda - u}_{L^{s_0} \parens{\Sigma_\lambda^u}}
\]
for every $\lambda \leq \lambda_2$. That is,
\[
\norm{u_\lambda - u}_{L^{s_0} \parens{\Sigma_\lambda^u}}
=
\norm{v_\lambda - v}_{L^{t_0} \parens{\Sigma_\lambda^v}}
=
0
\]
for every $\lambda \leq \lambda_2$. As such, \eqref{proof:0} holds with $M := |\lambda_2|$.

\medbreak

\noindent \textbf{Step 2:} \emph{Existence of hyperplane of symmetry.}
We want to prove that there exists
$\lambda_0 \in \real$ such that
\begin{equation}
\label{proof:2}
u = u_{\lambda_0}
\quad \text{and} \quad
v = v_{\lambda_0}
\quad \text{in} \quad
\Sigma_{\lambda_0}.
\end{equation}
Consider the set
\[
\mathcal{A}
:=
\left\{
	\lambda \in \real:
	u \geq u_\lambda
	~ \text{in} ~ \Sigma_\lambda, ~
	v \geq v_\lambda
	~ \text{in} ~ \Sigma_\lambda
	~ \text{and \eqref{eqn:or-condition} is satisfied}
\right\},
\]
where
\begin{equation}
\label{eqn:or-condition}
u|_{\Sigma_\lambda}
\neq
u_\lambda|_{\Sigma_\lambda}
\quad \text{or} \quad
v|_{\Sigma_\lambda}
\neq
v_\lambda|_{\Sigma_\lambda}.
\end{equation}
Due to \eqref{proof:-1}, we know that there exists
$\lambda \in \real$ such that
$u \geq u_\lambda$ and $v \geq v_\lambda$ in $\Sigma_\lambda$. The equality $\mathcal{A} = \emptyset$ implies
$u = u_\lambda, v = v_\lambda$ in $\Sigma_\lambda$, in which case \eqref{proof:2} is satisfied with $\lambda_0 = \lambda$. As such, we will henceforth suppose that $\mathcal{A} \neq \emptyset$. Let us show that $\lambda_0 := \sup \mathcal{A} < \infty$ and \eqref{proof:2} is satisfied.

\paragraph{Proof that $\lambda_0 < \infty$.}
By contradiction, suppose that $\lambda_0 = \infty$. Due to the Pigeonhole Principle, we can fix
$
\set{\lambda_n}_{n \in \nat} \subset \mathcal{A}
$
such that
$\lambda_n \to \infty$ and one of the following conditions is satisfied for every $n \in \nat$:
\begin{itemize}
\item
$u_{\lambda_n} \neq u$ in $\Sigma_{\lambda_n}$;
\item
$v_{\lambda_n} \neq v$ in $\Sigma_{\lambda_n}$.
\end{itemize}
Suppose, without loss of generality, that $u_{\lambda_n} \neq u$ in $\Sigma_{\lambda_n}$ for every $n \in \nat$. On one hand, it follows from the definition of $\mathcal{A}$ that
\begin{equation}
\label{eqn:limits}
\int_{\Sigma_{\lambda_n}} u \parens{x}^{s_0} \dif \mu
>
\int_{\Sigma_{\lambda_n}}
	u_{\lambda_n} \parens{x}^{s_0}
\dif \mu
=
\int_{\real^N \setminus \Sigma_{\lambda_n}}
	u \parens{x}^{s_0}
\dif \mu
\end{equation}
for every $n \in \nat$. On the other hand, due to the facts that
\begin{itemize}
\item
$u \in L^{s_0} (\real^N)$ and
\item
$\parens{\Sigma_{\lambda_n}}_{n \in \nat}$
is a nonincreasing sequence of sets such that
$\chi_{\Sigma_{\lambda_n}} \to 0$ as $n \to \infty$,
\end{itemize}
we obtain
\[
\int_{\Sigma_{\lambda_n}} u^{s_0} \dif \mu
\xrightarrow[n \to \infty]{}
0
\quad \text{and} \quad
\int_{\real^N \setminus \Sigma_{\lambda_n}} u^{s_0} \dif \mu
\xrightarrow[n \to \infty]{}
\norm{u}_{L^{s_0}}^{s_0}
>
0.
\]
These limits contradict \eqref{eqn:limits}, hence the result.

\paragraph{A simplifying assumption: $\lambda_0 = 0$.}
Let us show that we can suppose without loss of generality that
$\lambda_0 = 0$. Indeed, consider the functions $U, V \colon \real^N \to [0, \infty[$ defined as
\[
U \parens{x_1, x_2, \ldots, x_N}
=
u \parens*{\parens{x^{\frac{\lambda_0}{2}}}^0}
=
u \parens{x_1 - \lambda_0, x_2, \ldots, x_N}
\]
and
\[
V \parens{x_1, x_2, \ldots, x_N}
=
v \parens*{\parens{x^{\frac{\lambda_0}{2}}}^0}
=
v \parens{x_1 - \lambda_0, x_2, \ldots, x_N}.
\]
System \eqref{eqn:Hartree-system:2} is invariant by translation, so $\parens{U, V}$ is still a solution to \eqref{eqn:Hartree-system:2}. Consider the set
\[
\mathcal{B}
:=
\set*{
	\lambda \in \real:
	\norm{U_\lambda - U}_
		{L^{s_0} \parens{\Sigma_\lambda^u}}
	=
	\norm{V_\lambda - V}_
		{L^{t_0} \parens{\Sigma_\lambda^v}}
	=
	0
	~ \text{and \eqref{eqn:or-condition:2} is satisfied}
},
\]
where
\begin{equation}
\label{eqn:or-condition:2}
\norm{U - U_\lambda}_
	{L^{s_0} \parens{\Sigma_\lambda \setminus \Sigma_\lambda^u}}
>
0
\quad \text{or} \quad
\norm{V - V_\lambda}_
	{L^{t_0} \parens{\Sigma_\lambda \setminus \Sigma_\lambda^v}}
>
0\, .
\end{equation}
Let us prove that
$
\mathcal{B}
=
\set{
	\lambda - \lambda_0:
	\lambda \in \mathcal{A}
}
$,
from which we deduce that $\sup \mathcal{B} = 0$.

\noindent {\bf Proof that
$
\mathcal{B}
\subset
\set{\lambda - \lambda_0 : \lambda \in \mathcal{A}}
$.
}
Suppose that $\mu \in \mathcal{B}$. Given $x \in \Sigma_\mu$,
\begin{align*}
u \parens{x_1 - \lambda_0, x_2, \ldots, x_N}
=
U \parens{x}
\geq
U_\mu \parens{x}
&=
U \parens{x^\mu};
\\
&=
u \parens{2 \mu - x_1 - \lambda_0, x_2, \ldots, x_N};
\\
&=
u \parens*{
	2 \parens{\mu - \lambda_0} - \parens{x_1 - \lambda_0},
	x_2, \ldots, x_N
};
\\
&=
u_{\mu - \lambda_0} \parens{x_1 - \lambda_0, x_2, \ldots, x_N}.
\end{align*}
It follows that $u \geq u_{\mu - \lambda_0}$ in
$\Sigma_{\mu - \lambda_0}$. One similarly shows that
$v \geq v_{\mu - \lambda_0}$ in $\Sigma_{\mu - \lambda_0}$.
By hypothesis, either $U \neq U_\mu$ or $V \neq V_\mu$ in
$\Sigma_\mu$. Suppose without loss of generality that
$x_\mu \in \Sigma_\mu$ is such that $U (x_\mu) \neq U_\mu (x_\mu)$. That is,
\[
u (x_\mu - \lambda_0)
\neq
U (2 \mu - x_{\mu, 1}, x_{\mu, 2}, \ldots, x_{\mu, N})
=
u \parens*{
	2 (\mu - \lambda_0) - (x_{\mu, 1} - \lambda_0),
	x_{\mu, 2},
	\ldots,
	x_{\mu, N}
},
\]
from which we deduce that $u \neq u_{\mu - \lambda_0}$ in
$\Sigma_{\mu - \lambda_0}$. Finally, we obtain
$\mu - \lambda_0 \in \mathcal{A}$.

\noindent{\bf Proof that
$
\mathcal{B}
\supset
\set{\lambda - \lambda_0 : \lambda \in \mathcal{A}}
$.
}
Analogous to the proof of the reverse inclusion.
\paragraph{Proof that \eqref{proof:2} is satisfied.}
It suffices to prove that
\begin{equation}
\label{eqn:the-condition}
u \geq u_0, v \geq v_0
\quad \text{in} \quad \Sigma_0
\quad \text{and} \quad 0 \not \in \mathcal{A}.
\end{equation}
Indeed, if \eqref{eqn:the-condition} is satisfied, then \eqref{eqn:or-condition} is not satisfied because
$0 \not \in \mathcal{A}$. We deduce that $u = u_0$ and $v = v_0$ in $\Sigma_0$.

\noindent{\bf Proof that $u \geq u_0$, $v \geq v_0$ in $\Sigma_0$.}
Consider the function
$\delta \colon ]- \infty, 0] \times \real^N \to \real$
defined as
\[\delta (\lambda, x) = u \parens{x} - u_\lambda \parens{x}.\]
On one hand, the function $\delta$ is continuous, so
$\delta^{- 1} ([0, \infty[)$ is a closed subset of
$]- \infty, 0] \times \real^N$. On the other hand, the inclusion
\[
\mathcal{B}_{\mathcal{A}}
:=
\set*{\set{\lambda} \times \Sigma_\lambda: \lambda \in \mathcal{A}}
\subset
\delta^{- 1} \parens*{[0, \infty[}
\]
holds by definition. As such, we deduce that
\[
\set{0} \times \Sigma_0
\subset
\overline{\mathcal{B}_\mathcal{A}}
\subset
\delta^{- 1} \parens*{[0, \infty[},
\]
that is, $u \geq u_0$ in $\Sigma_0$. An analogous argument shows that $v \geq v_0$ in $\Sigma_0$.

\noindent{\bf Proof that $0 \not \in \mathcal{A}$.}
We only have to show that the following implication is satisfied:
\begin{equation}
\label{proof:3}
\text{if} ~ \lambda' \in \mathcal{A},
~ \text{then there exists} ~
\eps > 0
~ \text{such that} ~
\coi{\lambda', \lambda' + \eps} 
\subset
\mathcal{A}.
\end{equation}
Indeed, if \eqref{proof:3} holds, then the membership statement
$0 \in \mathcal{A}$ contradicts the fact that
$\sup \mathcal{A} = \lambda_0 = 0$. We proceed to the proof of \eqref{proof:3}. Let us show that there exists $\eps > 0$ such that if $\lambda' < \lambda < \lambda ' + \eps$, then
\begin{itemize}
\item
$
u|_{\Sigma_\lambda}
\neq
u_\lambda|_{\Sigma_\lambda}
$
or
$
v|_{\Sigma_\lambda}
\neq
v_\lambda|_{\Sigma_\lambda}
$
(\emph{first condition});
\item
$u \geq u_\lambda$ and $v \geq v_\lambda$ in $\Sigma_\lambda$
(\emph{second condition}).
\end{itemize}

\noindent{\bf First condition.}
Suppose, without loss of generality, that
$
u|_{\Sigma_{\lambda '}}
\neq
u_{\lambda'}|_{\Sigma_{\lambda '}}
$.
In particular, there exists $x_{\lambda '} \in \Sigma_{\lambda '}$ such that
$
u \parens{x_{\lambda '}}
>
u_{\lambda '}
\parens{x_{\lambda '}}
$.
The function $u$ is continuous, so there exists
$\eps > 0$ such that
$u \parens{x_{\lambda '}}
>
u_{\lambda} \parens{x_{\lambda '}}
$
whenever
$
\lambda'
<
\lambda
<
\lambda ' + \eps
$,
hence the result.

\noindent{\bf Preliminaries for the second condition.}
First, we want to show that
\begin{equation}
\label{proof:4}
u > u_{\lambda'}, v > v_{\lambda'}
\quad \text{in} \quad
\interior \Sigma_{\lambda'}.
\end{equation}
On one hand, Lemma \ref{lem:u_lambda-u} shows that
\begin{align*}
u \parens{x} - u_{\lambda'} \parens{x}
&=
\frac{2 p}{p + q}
\int_{\Sigma_{\lambda'}}
	\parens*{
		G_2^\tau \parens{x - y}
		-
		G_2^\tau (x^{\lambda'} - y)
	}	\parens*{
		\phi_v \parens{y}
		u \parens{y}^{p - 1}
		-
		\phi_{v, {\lambda'}} \parens{y}
		u_{\lambda'} \parens{y}^{p - 1}
	}
\dif y
\end{align*}
and
\[
\phi_v \parens{x} - \phi_{v, {\lambda'}} \parens{x}
=
\int_{\Sigma_{\lambda'}}
	\parens*{I_\alpha \parens{x - y} - I_\alpha (x^{\lambda'} - y)}
	\parens*{
		v \parens{y}^q
		-
		v_{\lambda'} \parens{y}^q
	}
\dif y
\]
for every $x \in \Sigma_{\lambda'}$. On the other hand, Lemma \ref{lem:preliminary-inequalities} shows that
\[
G_2^\tau \parens{x - y}
-
G_2^\tau \parens{x^{\lambda'} - y}
>
0
\quad \text{and} \quad
I_\alpha \parens{x - y}
-
I_\alpha \parens{x^{\lambda'} - y}
>
0
\]
for every
$x \in \interior \Sigma_{\lambda'}$
and
$y \in \interior \Sigma_{\lambda'} \setminus \set{x}$.
As such, we only have to show that
\begin{equation}
\label{eqn:sufficient-inequality}
\phi_v
u ^{p - 1}
-
\phi_{v, {\lambda'}}
u_{\lambda'} ^{p - 1}
>
0
\quad
\text{in a nonempty open subset of}
\quad
\Sigma_{\lambda'}
\end{equation}
to deduce that
$u > u_{\lambda'}$ in $\interior \Sigma_{\lambda'}$.
Indeed, \eqref{eqn:sufficient-inequality} holds because the membership statement $\lambda' \in \mathcal{A}$ implies
$u > u_{\lambda'}$ or $v > v_{\lambda'}$ in a nonempty open subset of $\Sigma_{\lambda'}$. The proof that $v > v_{\lambda'}$ in
$\interior \Sigma_{\lambda'}$ is analogous.

Consider the function
$\chi_{\Sigma_{\lambda' + \eps}^u} \colon \real^N \to \set{0, 1}$
defined as
\[
\chi_{\Sigma_{\lambda' + \eps}^u}
\parens{x}
=
\begin{cases}
1,
&\text{if} ~
u_{\lambda' + \eps} \parens {x}
>
u \parens{x}
~ \text{and} ~
x_1 \geq \lambda' + \eps;
\\
0,
&\text{if} ~
u_{\lambda' + \eps} \parens {x}
\leq
u \parens{x}
~ \text{and} ~
x_1 \geq \lambda' + \eps;
\\
0,
&\text{if} ~
x_1 < \lambda' + \eps.
\end{cases}
\]
Let us show that
$\chi_{\Sigma_{\lambda' + \eps}^u}$,
$\chi_{\Sigma_{\lambda' + \eps}^v} \to 0$ as
$\eps \to 0^+$. Suppose that $x \in \real^N$ is such that
$\chi_{\Sigma_{\lambda' + \eps}^u} \parens{x}$
does not converge to zero. In particular, there exist $\delta > 0$ and
$
\set{\eps_n}_{n \in \nat}
\subset
\ooi{0, \infty}
$
such that
$
\chi_{\Sigma_{\lambda' + \eps_n}^u} \parens{x}
\geq
\delta
$
for every $n \in \nat$ and $\eps_n \to 0$ as
$n \to \infty$. On one hand, $x \in \interior \Sigma_{\lambda'}$
because there exists $n \in \nat$ such that
$\chi_{\Sigma_{\lambda' + \eps_n}^u} \parens{x} = 1 $.
On the other hand, the continuity of $u$ implies
\[
1
=
\lim_{n \to \infty} u_{\lambda' + \eps_n} \parens{x}
=
u_{\lambda'} \parens{x}
\geq
u \parens{x}.
\]
We just obtained a contradiction with \eqref{proof:4}, hence the result.

Next, we want to show that
\begin{equation}
\label{proof:5}
\norm{u_{\lambda' + \eps}}_{
	L^{s_j} \parens{\Sigma_{\lambda' + \eps}^u}
},
~
\norm{v_{\lambda' + \eps}}_{
	L^{t_j} \parens{\Sigma_{\lambda' + \eps}^v}
},
~
\norm{\psi_{u, \lambda' + \eps}}_{
	L^{\gamma_j} \parens{\Sigma_{\lambda' + \eps}^u}
},
~
\norm{\phi_{v, \lambda' + \eps}}_{
	L^{\beta_j} \parens{\Sigma_{\lambda' + \eps}^v}
}
\xrightarrow[\eps \to 0^+]{}
0.
\end{equation}
These limits are proved similarly, so we will only show that
$
\norm{v_{\lambda' + \eps}}_{
	L^{t_j} \parens{\Sigma_{\lambda' + \eps}^v}
}
\to
0
$
as $\eps \to 0^+$. Clearly,
\[
\norm{v_{\lambda' + \eps}}_{
	L^{t_j} \parens{\Sigma_{\lambda' + \eps}^v}
}^{t_j}
=
\int_{\Sigma_{\lambda' + \eps}^v}
	v_{\lambda'}^{t_j}
\dif \mu
+
\int_{\Sigma_{\lambda' + \eps}^v}
	v_{\lambda' + \eps}^{t_j}
	-
	v_{\lambda'}^{t_j}
\dif \mu.
\]
The first integral on the right-hand side tends to zero as
$\eps \to 0^+$ by dominated convergence because
$
v_{\lambda'}^{t_j}
\chi_{\Sigma_{\lambda' + \eps}^v}
\to
0
$
pointwise as $\eps \to 0^+$ and
\[
0
\leq
v_{\lambda'}^{t_j}
\chi_{\Sigma_{\lambda' + \eps}^v}
\leq
v_{\lambda'}^{t_j} \in L^1 
\]
for every $\eps \in \ooi{0, 1}$. Now, we consider the second integral in the right-hand side. The term
\[
\frac{
	v_{\lambda' + \eps} \parens{x}^{t_j}
	-
	v_{\lambda'} \parens{x}^{t_j}
}{2 \eps}
=
\frac{
	v \parens*{2 (\lambda' + \eps) - x_1, x_2, \ldots, x_N}^{t_j}
	-
	v \parens{2 \lambda' - x_1, x_2, \ldots, x_N}^{t_j}
}{2 \eps}
\]
is a Newton quotient. We can use the Cauchy--Schwarz inequality to verify that
\[
\Sigma_{\lambda'} \ni x
\mapsto
t_j
v (x^{\lambda'})^{t_j - 1}
\partial_1 v (x^{\lambda'})
\in
\real
\]
is integrable. Indeed,
$v \in L^{2 \parens{t_j - 1}} \parens{\real^N}$
because $t_j \geq 2$ and it follows from \cite[Theorem 1.4]{boerStandingWavesNonlinear2025} that there exists arbitrarily large $\overline{t} > 1$ such that
$v \in L^{\overline{t}} \parens{\real^N}$.
By dominated convergence, we obtain
\[
\int_{\Sigma_{\lambda' + \eps}^v}
	v_{\lambda' + \eps}^{t_j}
	-
	v_{\lambda'}^{t_j}
\dif \mu
=
2 \eps
\int_{\Sigma_{\lambda' + \eps}^v}
	\frac{
		v_{\lambda' + \eps}^{t_j}
		-
		v_{\lambda'}^{t_j}
	}{2 \eps}
\dif \mu
\xrightarrow[\eps \to 0^+]{}
0
\]
(see \cite[Theorem 2.27]{follandRealAnalysisModern1999}).

\noindent{\bf Second condition.}
Due to the limits in \eqref{proof:5}, we can fix $\eps > 0$ such that \eqref{proof:0.5}, \eqref{proof:1.5} are satisfied for every
$
\lambda \in ]\lambda', \lambda' + \eps[
$.
As such, we can argue as in the proof of \eqref{proof:0} to deduce that
\[
\norm{u_\lambda - u}_{L^{s_0} (\Sigma_\lambda^u)}
=
\norm{v_\lambda - v}_{L^{t_0} (\Sigma_\lambda^v)}
=
0
\]
for every
$\lambda \in \ooi{\lambda', \lambda' + \eps}$.

\medbreak
\noindent \textbf{Step 3:} \emph{Conclusion.} To finish, we can repeat the argument to show that, up to translation, $u, v$ are symmetric with respect to any hyperplane of $\real^N$ passing through the origin, and hence radially symmetric with respect to the origin, and radially decreasing.  Then, applying the Hopf Lemma to both equations in \eqref{eqn:Hartree-system:2}, it follows that $u$ and $v$ are strictly radially decreasing.
\end{proof}

\section{Classification of positive ground states}
\label{classification}

A change of variable shows that if
$u, v \colon \real^N \to \real$
solve
\begin{equation}
\label{eqn:Choquard-system:2}
\begin{cases}
- \Delta u + u
=
\parens*{I_\alpha \ast \abs{v}^p}
\abs{u}^{p - 2} u
&\text{in} ~ \real^N,
\\
- \Delta v + v
=
\parens*{I_\alpha \ast \abs{u}^p}
\abs{v}^{p - 2} v
&\text{in} ~ \real^N,
\end{cases}
\end{equation}
then the pair
$
\tau^{\frac{\alpha + 2}{4 \parens{p - 1}}}
\parens{
	u \parens{\sqrt{\tau} \cdot},
	v \parens{\sqrt{\tau} \cdot}
}
$
solves \eqref{eqn:Choquard-system}. As such, we will henceforth suppose without loss of generality that $\tau = 1$. Let us introduce a minimization problem that is closely related to \eqref{eqn:min-N_A}, i.e.,
\begin{equation}
\label{eqn:min-Q}
\begin{cases}
Q \parens{u}
=
\mathcal{S}
:=
\inf_{
	v \in H^1 \parens{\real^N} \setminus \set{0}
}
Q \parens{v},
\\
u \in H^1 \parens{\real^N} \setminus \set{0},
\end{cases}
\end{equation}
where the nonlinear functional
$
Q
\colon
H^1 \parens{\real^N} \setminus \set{0}
\to
\ooi{0, \infty}
$
is defined as
\[
Q \parens{u}
=
\frac{
		\int_{\real^N}
			\abs{\nabla u}^2 + u^2
		\dif \mu
	}{
		\parens*{
			\int_{\real^N}
				\parens*{I_\alpha \ast \abs{u}^p}
				\abs{u}^p
			\dif \mu
		}^{\frac{1}{p}}
	}\, .
\]
We proceed to a discussion about how \eqref{eqn:min-N_A} and \eqref{eqn:min-Q} compare to each other.

\begin{rmk}
\label{rmk:MorozVanSchaftingen}
It follows from \cite[Proposition 2.1]{morozGroundstatesNonlinearChoquard2013} that
\begin{itemize}
\item
$
\mathcal{T}
=
\parens*{\frac{1}{2} - \frac{1}{2 p}}
\mathcal{S}^{\frac{p}{p - 1}}
$,
where we recall that $\T$ is the action level defined in \eqref{eqn:min-N_A};
\item
if $u$ solves \eqref{eqn:min-N_A}, then $u$ solves \eqref{eqn:min-Q};
\item
if $u$ solves \eqref{eqn:min-Q}, then $k u$ solves \eqref{eqn:min-N_A}, where $k$ denotes the unique positive constant such that
$k u \in \N_\A$, that is,
\[
k
=
\parens*{
\frac{\norm{u}_{H^1}^2}{
	\int_{\real^N}
		\parens*{I_\alpha \ast \abs{u}^p}
		\abs{u}^p
	\dif \mu
}}^{\frac{1}{2 \parens{p - 1}}}.
\]
\end{itemize}
\end{rmk}

Consider the functionals
$F_1, F_2 \colon E \parens{\real^N} \to \real$
defined as
\[
F_1 \parens{u, v}
=
\norm{u}_{H^1}^2
-
\int_{\real^N}
	\parens*{I_\alpha \ast \abs{v}^p} \abs{u}^p
\dif \mu
\quad \text{and} \quad
F_2 \parens{u, v}
=
\norm{v}_{H^1}^2
-
\int_{\real^N}
	\parens*{I_\alpha \ast \abs{u}^p} \abs{v}^p
\dif \mu.
\]
The motivation for these definitions is that if $F_i \parens{u, v} = 0$, then $\parens{u, v}$ satisfies the Nehari identity associated with the $i$\textsuperscript{th} equation in \eqref{eqn:Choquard-system:2}. We also introduce the following Nehari-type manifold:
\[
\widetilde{\N}
:=
\set*{
	\parens{u, v} \in E \parens{\real^N}:
	u, v \not \equiv 0;
	\quad
	F_1 \parens{u, v} = 0
	\quad \text{and} \quad
	F_2 \parens{u, v} = 0
}
\]
and we denote its least action level by
$
\widetilde{c}
:=
\inf_{\parens{u, v} \in \widetilde{\N}}
J \parens{u, v}
$.
Let us show that the newly-defined action level
$\widetilde{c}$ coincides with the action level $c$ defined in \eqref{eqn:min-N}.

\begin{lem}
\label{lem:action-lvl-Nehari}
The equality 
$c = \widetilde{c}$ holds.
\end{lem}
\begin{proof}
Due to the inclusion
$\widetilde{\N} \subset \N_J$,
it holds that
$
\inf_{\widetilde{\N}} J = \widetilde{c}
\geq
c = \inf_{\N_J} J
$.
As such, we only have to show that
$\widetilde{c} \leq c$.
Let $\parens{u, v} \in \N_J$ denote a ground state of \eqref{eqn:Choquard-system:2} (whose existence follows from \cite[Theorem 1.1]{boerStandingWavesNonlinear2025}). In particular, $\parens{u, v}$ is a classical solution due to \cite[Theorem 1.4]{boerStandingWavesNonlinear2025}. Therefore, the Nehari identity associated with each equation in \eqref{eqn:Choquard-system:2} is satisfied, that is,
$
F_1 \parens{u, v} = F_2 \parens{u, v} = 0
$.
It also follows from
\cite[Theorem 1.1]{boerStandingWavesNonlinear2025}
that
$u \not \equiv 0$ and $v \not \equiv 0$,
so
$
\parens{u, v} \in \widetilde{\mathcal{N}}
$.
We deduce that
$\widetilde{c} \leq J \parens{u, v} = c$.
\end{proof}

The following result is inspired by \cite[Lemma 2.2]{wangStandingWavesCoupled2017}. More precisely, it establishes the Cauchy--Schwarz inequality associated with a certain inner product on
$L^{\frac{2 N}{N + \alpha}} \parens{\real^N}$.
\begin{lem}
\label{lem:CS}
\begin{enumerate}
\item
Given
$
u, v
\in
L^{\frac{2 N}{N + \alpha}} \parens{\real^N}
$,
the following inequalities are satisfied:
\[
\int_{\real^N} \int_{\real^N}
	I_\alpha \parens{x - y}
	u \parens{x}
	u \parens{y}
\dif x \dif y
\geq 0,
\quad
\int_{\real^N} \int_{\real^N}
	I_\alpha \parens{x - y}
	v \parens{x}
	v \parens{y}
\dif x \dif y
\geq
0
\]
and
\begin{equation}
\label{lem:CS:0}
\begin{aligned}
\abs*{
\int_{\real^N} \int_{\real^N}
	I_\alpha \parens{x - y}
	u \parens{x}
	v \parens{y}
\dif x \dif y
}
&\leq
\parens*{
	\int_{\real^N} \int_{\real^N}
		I_\alpha \parens{x - y}
		u \parens{x}
		u \parens{y}
	\dif x \dif y
}^{\frac{1}{2}}
\\
&\times
\parens*{
	\int_{\real^N} \int_{\real^N}
		I_\alpha \parens{x - y}
		v \parens{x}
		v \parens{y}
	\dif x \dif y
}^{\frac{1}{2}}.
\end{aligned}
\end{equation}
\item
Suppose that
$
u, v
\in
L^{\frac{2 N}{N + \alpha}} \parens{\real^N}
$.
The equality case in \eqref{lem:CS:0} holds if, and only if, there exists $\lambda \in \real$ such that $v = \lambda u$.
\end{enumerate}
\end{lem}
\begin{proof}
Notice that the integrals in the statement are well defined due to the Hardy--Littlewood--Sobolev inequality.

\paragraph{A preliminary identity:} There exists a positive constant $K$ such that
\begin{equation}
\label{lem:CS:1}
\frac{1}{\abs{x - y}^{N - \alpha}}
=
K
\int_{\real^N}
	\frac{\dif z}{
		\abs{x - z}^{\frac{2 N - \alpha}{2}}
		\abs{y - z}^{\frac{2 N - \alpha}{2}}
	}
\end{equation}
for every $x \in \real^N$ and
$y \in \real^N \setminus \set{x}$.
Notice that the integral on the right-hand side is well defined because
$0 < \alpha < N$.
Define
\[
h \parens{x, y}
=
\int_{\real^N}
	\frac{\dif z}{
		\abs{x - z}^{\frac{2 N - \alpha}{2}}
		\abs{y - z}^{\frac{2 N - \alpha}{2}}
	}
\]
for every
$\parens{x, y} \in \real^N \times \real^N$ with $x \neq y$. A change of variable shows that
\[
h \parens{x, y}
=
h \parens{0, y - x}
=
\int_{\real^N}
	\frac{\dif z}{
		\abs{z}^{\frac{2 N - \alpha}{2}}
		\abs{y - x - z}^{\frac{2 N - \alpha}{2}}
	}
\]
and $h \parens{0, R y} = h \parens{0, y}$ for every $R \in \mathrm{O} \parens{N}$. We deduce that there exists
$
H
\colon
\ooi{0, \infty} \to \ooi{0, \infty}
$
such that
$h \parens{x, y} = H \parens{\abs{x - y}}$.
Another change of variable shows that
$
H \parens{t s}
=
\frac{H \parens{s}}{t^{N - \alpha}}
$
for every $s, t > 0$. Therefore, the right-hand side of \eqref{lem:CS:1} scales similarly as the left-hand side, hence the result.

\paragraph{Proof of the inequality.}
It follows from Fubini's theorem and \eqref{lem:CS:1} that
\[
\int_{\real^N} \int_{\real^N}
	\frac{
		u \parens{x}
		v \parens{y}
	}{
		\abs{x - y}^{N - \alpha}
	}
\dif x \dif y
=
K
\int_{\real^N}
	f_u \parens{z}
	f_v \parens{z}
\dif z,
\]
where
\[
f_g \parens{z}
:=
\parens*{
	\int_{\real^N}
		\frac{g \parens{x}}{
			\abs{x - z}^{\frac{2 N - \alpha}{2}}
		}
	\dif x
}.
\]
Due to the Cauchy--Schwarz inequality,
\begin{equation}
\label{lem:CS:1.5}
\abs*{
\int_{\real^N} \int_{\real^N}
	\frac{
		u \parens{x}
		v \parens{y}
	}{
		\abs{x - y}^{N - \alpha}
	}
\dif x \dif y
}
\leq
K
\parens*{
\int_{\real^N}
	f_u \parens{z}^2
\dif z
}^{\frac{1}{2}}
\parens*{
\int_{\real^N}
	f_v \parens{z}^2
\dif z
}^{\frac{1}{2}}
=
\eqref{lem:CS:2},
\end{equation}
where
\begin{multline}
\label{lem:CS:2}
K
\parens*{
\int_{\real^N}
\parens*{
	\int_{\real^N}
		\frac{u \parens{x}}{
			\abs{x - z}^{\frac{2 N - \alpha}{2}}
		}
	\dif x
}
\parens*{
	\int_{\real^N}
		\frac{u \parens{y}}{
			\abs{y - z}^{\frac{2 N - \alpha}{2}}
		}
	\dif x
}
\dif z
}^{\frac{1}{2}}
\times
\\
\times
\parens*{
\int_{\real^N}
\parens*{
	\int_{\real^N}
		\frac{v \parens{x}}{
			\abs{x - z}^{\frac{2 N - \alpha}{2}}
		}
	\dif x
}
\parens*{
	\int_{\real^N}
		\frac{v \parens{y}}{
			\abs{y - z}^{\frac{2 N - \alpha}{2}}
		}
	\dif x
}
\dif z
}^{\frac{1}{2}}.
\end{multline}
At this point, the inequality follows from an application of \eqref{lem:CS:1} to \eqref{lem:CS:2}. Observe that the integrals at the right-hand side in \eqref{lem:CS:0} are nonnegative due to the identity
\[
\int_{\real^N} \int_{\real^N}
	I_\alpha \parens{x - y}
	g \parens{x}
	g \parens{y}
\dif x \dif y
=
\int_{\real^N}
	f_g \parens{z}^2
\dif z.
\]

\paragraph{Saturation of the inequality.}
The saturation of the inequality only happens when $v = \lambda u$ because this is the precisely the situation where the equality in \eqref{lem:CS:1.5} is satisfied.
\end{proof}

We proceed to the proof of the theorem.

\begin{proof}[\textbf{Proof of Theorem \ref{thm:classification}}]
Let $w \in H^1 \parens{\real^N}$ denote a (positive) ground state of the Choquard-type equation
\[
- \Delta w +  w
=
\parens*{I_\alpha \ast \abs{w}^p}
\abs{w}^{p - 2} w,
\]
whose existence follows from \cite[Theorem 3.2]{morozGuideChoquardEquation2017}. Due to the fact that
$\parens{w, w} \in \widetilde{\N}$, it follows that
\begin{equation}
\label{proof:classification:1}
\widetilde{c}
\leq
J \parens{w, w}
=
2 \mathcal{T}
=
\frac{p - 1}{p}
\mathcal{S}^\frac{p}{p - 1}.
\end{equation}
Let
\[
a
=
\int_{\real^N}
	\parens*{I_\alpha \ast \abs{u}^p} \abs{u}^p
\dif \mu
\quad \text{and} \quad
b
=
\int_{\real^N}
	\parens*{I_\alpha \ast \abs{v}^p} \abs{v}^p
\dif \mu.
\]
In view of Lemma \ref{lem:CS} and the fact that
$\parens{u, v} \in \widetilde{\mathcal{N}}$, one infers that
\begin{equation}
\label{proof:classification:2}
\mathcal{S} a^{\frac{1}{p}}
\leq
\norm{u}_{H^1}^2
=
\int_{\real^N}
	\parens*{I_\alpha \ast \abs{u}^p} \abs{v}^p
\dif \mu
\leq
a^{\frac{1}{2}} b^{\frac{1}{2}}
\end{equation}
and
\begin{equation}
\label{proof:classification:3}
\mathcal{S} b^{\frac{1}{p}}
\leq
\norm{v}_{H^1}^2
=
\int_{\real^N}
	\parens*{I_\alpha \ast \abs{v}^p} \abs{u}^p
\dif \mu
\leq
a^{\frac{1}{2}} b^{\frac{1}{2}}.
\end{equation}
It follows from \eqref{proof:classification:2} and \eqref{proof:classification:3} that
\[
a^{\parens*{\frac{1}{2} - \frac{1}{p}}}
b^{\frac{1}{2}}
\geq \mathcal{S}
\quad \text{and} \quad
a^{\frac{1}{2}}
b^{\parens*{\frac{1}{2} - \frac{1}{p}}}
\geq \mathcal{S},
\]
which imply
\begin{equation}
\label{proof:classification:4}
a^{\frac{p - 1}{p}}
b^{\frac{p - 1}{p}}
\geq
\mathcal{S}^2.
\end{equation}
Then, from Lemma \ref{lem:action-lvl-Nehari} and \eqref{proof:classification:1}--\eqref{proof:classification:3},
\[
\mathcal{S}
\parens*{a^{\frac{1}{p}} + b^{\frac{1}{p}}}
\leq
\norm*{\parens{u, v}}_E^2
=
\frac{2 p}{p - 1}
c
=
\frac{2 p}{p - 1}
\widetilde{c}
\leq
2 \mathcal{S}^{\frac{p}{p - 1}},
\]
and thus
\begin{equation}
\label{proof:classification:5}
a^{\frac{1}{p}} + b^{\frac{1}{p}}
\leq
2 \mathcal{S}^{\frac{1}{p - 1}}.
\end{equation}

We claim that
\begin{equation}
\label{proof:classification:6}
a = b = \mathcal{S}^{\frac{p}{p - 1}}.
\end{equation}
Indeed, consider the following regions of
$
P
:=
\oci{0, 2 \mathcal{S}^{\frac{1}{p - 1}}}^2
\subset
\real^2
$:
\[
\Gamma_-
=
\set*{
	\parens{s, t} \in P:
	s + t \leq 2 \mathcal{S}^{\frac{1}{p - 1}}
}
\quad \text{and} \quad
\Gamma_+
=
\set*{
	\parens{s, t} \in P:
	s^{p - 1} t^{p - 1}
	\geq
	\mathcal{S}^2
}.
\]
The boundary of these regions are parametrized by the functions
$
f_\pm
\colon
\oci{0, 2 \mathcal{S}^{\frac{1}{p - 1}}}
\to
\ooi{0, \infty}
$
respectively defined as
\[
f_- \parens{s}
=
2 \mathcal{S}^{\frac{1}{p - 1}} - s
\quad \text{and} \quad
f_+ \parens{s}
=
\mathcal{S}^{\frac{2}{p - 1}} s^{- 1}.
\]
Let $f = f_+ - f_-$. On one hand, it is immediate that
$f \parens{\mathcal{S}^{\frac{1}{p - 1}}} = 0$.
On the other hand, it follows from single-variable calculus that
\[
\begin{cases}
f' \parens{s} < 0,
&
\text{if}
~
0 < s < \mathcal{S}^{\frac{1}{p - 1}}\, ,
\\
f ' \parens{\mathcal{S}^{\frac{1}{p - 1}}}
=
0\, ,
\\
f' \parens{s} > 0,
&
\text{if}
~
s > \mathcal{S}^{\frac{1}{p - 1}}.
\end{cases}
\]
Then, $f \parens{s} = 0$ if, and only if $s = \mathcal{S}^{\frac{1}{p - 1}}$. As such, \eqref{proof:classification:4} and \eqref{proof:classification:5} imply
\[
\parens{a^{\frac{1}{p}}, b^{\frac{1}{p}}}
\in
\Gamma_- \cap \Gamma_+
=
\set*{
	\parens{
		\mathcal{S}^{\frac{1}{p - 1}},
		\mathcal{S}^{\frac{1}{p - 1}}
	}
},
\]
hence the result.

In view of \eqref{proof:classification:6}, it follows from \eqref{proof:classification:2} and \eqref{proof:classification:3} that
\[
\frac{\norm{u}_{H^1}^2}{a^{\frac{1}{p}}}
=
\frac{\norm{v}_{H^1}^2}{b^{\frac{1}{p}}}
=
\mathcal{S}.
\]
That is, $u$ and $v$ are positive solutions to \eqref{eqn:min-Q}. Due to the equalities
$a = \norm{u}_{H^1}^2$ and
$b = \norm{v}_{H^1}^2$, it follows that
$u, v \in \N_\A$. Moreover, in view of Remark \ref{rmk:MorozVanSchaftingen}, $u, v$ are positive ground states of \eqref{eqn:Choquard-type}. Due to the equalities
\[
a
=
b
=
\int_{\real^N}
	\parens*{I_\alpha \ast \abs{v}^p} \abs{u}^p
\dif \mu,
\]
it follows from Lemma \ref{lem:CS} that
$v = \lambda u$ for a certain
$\lambda \in \real$. As $a = b \neq 0$, it follows that $\lambda = 1$.
\end{proof}

\subsection*{Acknowledgment}
Eduardo de Souza B\"oer was supported by FAPESP/Brazil grant 2023/05445-2. Ederson Moreira dos Santos was partially supported by CNPq/Brazil grant 312867/2023-9 and FAPESP/Brazil grant 2022/16407-1. 
Gustavo de Paula Ramos was supported by FAPESP/Brazil grant 2024/20593-0.

\end{document}